\documentclass[oneside,a4paper]{amsart}

\usepackage[T1]{fontenc}
\usepackage[utf8]{inputenc}
\usepackage{amsmath, amsthm, amsfonts, amssymb}

\usepackage{etoolbox}
\usepackage{url}
\usepackage[all]{xy}
\usepackage{placeins} 
\usepackage[usenames, dvipsnames]{color} 
\usepackage{euscript} 
\usepackage{bbm} 
\usepackage{enumitem} 

\usepackage[dvipsnames]{xcolor} 
\definecolor{dark-red}{rgb}{0.5,0.15,0.15}
\definecolor{dark-blue}{rgb}{0.15,0.15,0.6}
\definecolor{dark-green}{rgb}{0.15,0.6,0.15}
\usepackage{hyperref}
\hypersetup{
    colorlinks, linkcolor=OliveGreen,
    citecolor=dark-blue, urlcolor=dark-green
}
\usepackage[nameinlink,capitalise,noabbrev]{cleveref}

\usepackage{tikz}
\usetikzlibrary{arrows,backgrounds,
    decorations.pathreplacing,
    decorations.pathmorphing
}
\usetikzlibrary{matrix,arrows}

\newcommand{\euscr}[1]{\EuScript{#1}} 
\newcommand{\acat}{\euscr{A}} 
\newcommand{\ccat}{\euscr{C}} 
\newcommand{\dcat}{\euscr{D}} 
\newcommand{\ecat}{\euscr{E}} 
\newcommand{\Fun}{\textnormal{Fun}} 
\newcommand{\Hom}{\textnormal{Hom}} 
\newcommand{\Ext}{\textnormal{Ext}} 
\newcommand{\map}{\textnormal{map}} 
\newcommand{\spectra}{\euscr{S}p} 

\newcommand{\synspectra}{\euscr{S}yn} 
\newcommand{\ComodE}{\euscr{C}omod_{E_{*}E}} 
\newcommand{\Comod}{\euscr{C}omod} 
\newcommand{\Mod}{\euscr{M}od} 
\newcommand{\monunit}{\mathbbm{1}} 
\newcommand{\monunitt}[1]{\monunit_{\leq #1}} 
\newcommand{\atrun}[1]{A_{\leq #1}} 
\newcommand{\ekoperad}[1]{\mathbf{E}_{#1}} 

\newcommand{\triplerightarrow}{%
\tikz[minimum height=0ex]
  \path[->]
   node (a)            {}
   node (b) at (1em,0) {}
  (a.north)  edge (b.north)
  (a.center) edge (b.center)
  (a.south)  edge (b.south);%
}

\theoremstyle{plain}

\newtheorem{thm}{Theorem}[section]
\newtheorem{lemma}[thm]{Lemma}
\newtheorem{prop}[thm]{Proposition}
\newtheorem{cor}[thm]{Corollary}

\newtheorem*{thm*}{Theorem}

\theoremstyle{definition}
\newtheorem{example}[thm]{Example}

\newtheorem{defin}[thm]{Definition}
\newtheorem{rem}[thm]{Remark}

\newtheorem*{rem*}{Remark}
\newtheorem*{interpretation*}{Interpretation}
\newtheorem*{defin*}{Definition}
\newtheorem*{conjecture*}{Conjecture}
\newtheorem*{notation*}{Notation}
\newtheorem*{convention*}{Convention}
\newtheorem*{thm_italics*}{Theorem}

\theoremstyle{remark}

\numberwithin{equation}{section}

\makeatletter
  \def\subsection{\@startsection{subsection}{1}%
  \z@{.7\linespacing\@plus\linespacing}{.5\linespacing}%
  {\normalfont\bfseries\centering}}
\makeatother

\let\oldtocsection=\tocsection
\let\oldtocsubsection=\tocsubsection
\let\oldtocsubsubsection=\tocsubsubsection
\renewcommand{\tocsection}[2]{\hspace{0em}\oldtocsection{#1}{#2}}
\renewcommand{\tocsubsection}[2]{\hspace{1em}\oldtocsubsection{#1}{#2}}
\renewcommand{\tocsubsubsection}[2]{\hspace{2em}\oldtocsubsubsection{#1}{#2}}

\calclayout

\newcommand{\ThetaSect}{\Theta{\text -}\mathrm{Sect}}
\newcommand{\ThetaMorph}{\Theta{\text -}\mathrm{Morph}}
\newcommand{\LL}{\mathbb{L}}
\DeclareMathOperator{\Aut}{Aut}
\DeclareMathOperator{\Spec}{Spec}
\DeclareMathOperator{\CAlg}{CAlg}
\DeclareMathOperator{\Alg}{Alg}

\begin{document}
\title[Abstract Goerss-Hopkins Theory]{Abstract Goerss-Hopkins Theory}
\author[Piotr Pstr\k{a}gowski, Paul VanKoughnett]{Piotr Pstr\k{a}gowski, Paul VanKoughnett}
\thanks{The first author was supported by the Danish National Research Foundation through the Copenhagen Centre for Geometry and Topology (DNRF151). The second author was supported by the National Science Foundation Grant No. 1440140, while he was in residence at the Mathematical Sciences Research Institute in Berkeley, California, during the spring semester of 2019, as well as under Grant No. 1714273.}
\address{Harvard University}
\email{pstragowski.piotr@gmail.com, pvankoug@math.northwestern.edu}

\begin{abstract}
We present an abstract version of Goerss-Hopkins theory in the setting of a prestable $\infty$-category equipped with a suitable periodicity operator. In the case of the $\infty$-category of synthetic spectra, this yields obstructions to realizing a comodule algebra as a homology of a commutative ring spectrum, recovering the results of Goerss and Hopkins. 
\end{abstract}

\maketitle 

\tableofcontents


\section{Introduction}

Goerss-Hopkins obstruction theory gives obstructions to realizing an algebra in $E_{*}E$-comodules as $E$-homology of a commutative ring spectrum \cite{moduli_spaces_of_commutative_ring_spectra}, \cite{moduli_problems_for_structured_ring_spectra}. 

A famous application of these results is the Goerss-Hopkins-Miller theorem, which states that Morava $E$-theory spectrum admits an essentially unique structure of an $\ekoperad{\infty}$-ring spectrum, which is strictly functorial in the underlying formal group law, producing an action of the Morava stabilizer group. This action, and the corresponding fixed point spectra constructed by Devinatz and Hopkins \cite{devinatz2004homotopy}, are essential in modern chromatic homotopy theory and have led to many results of both conceptual and computational power, such as the construction of the spectrum of topological modular forms \cite{goerss2005resolution}, \cite{rognes2005galois}, \cite{behrens2011notes},  \cite{behrens2011higher}, \cite{beaudrygoersshenn2017chromatic}.

In this note, we give an independent and more general account of Goerss-Hopkins theory, suitable for realizing objects from homological data in a wide class of stable $\infty$-categories. Moreover, our approach is also simpler and more direct, making it interesting even in the classical case, since the original papers \cite{moduli_spaces_of_commutative_ring_spectra}, \cite{moduli_problems_for_structured_ring_spectra} are well-known for their technicality. 

Roughly, one can divide our work into two parts, the first being a construction of an appropriate $\infty$-category of synthetic objects, where one can work with "resolutions". In the case relevant to realizing commutative ring spectra, one needs an appropriate theory of synthetic $E$-local spectra; these have been constructed by the first author in \cite{pstrkagowski2018synthetic}.

The second part, to which this note is devoted, takes such an $\infty$-category of "resolutions" as an input and defines an appropriate tower of moduli of potential $n$-stages, the definition of which goes back to the fundamental work of Blanc, Dwyer and Goerss \cite{realization_space_of_a_pi_algebra}. The obstruction theory is then derived from the detailed study of this tower; essentially, the obstructions are obtained from the deformation theory of algebras in synthetic objects.

\subsection{Existence of realizations} 

Suppose that $\ccat$ is a graded symmetric monoidal, complete Grothendieck prestable $\infty$-category; that is, $\ccat$ is the subcategory of connective objects in a well-behaved stable $\infty$-category. 

We say a commutative algebra $A \in \textnormal{CAlg}(\ccat)$ is \emph{shift} if we have a map $\tau: \Sigma A[-1] \rightarrow A$ of $A$-modules which induces an isomorphism $\pi_{*} A \simeq (\pi_{0} A)[\tau]$. We say an $A$-module $M$ is \emph{periodic} if $\pi_{*}M \simeq \pi_{*} A \otimes _{\pi_{0} A} \pi_{0} M \simeq (\pi_{0} M)[\tau]$ and that it is \emph{flat} if additionally $\pi_{0}M$ is flat over $\pi_{0} A$. 

Let $A$ be a shift algebra such that $\Mod_{A}(\ccat)$ is generated under colimits by flat modules. In this note, we present Goerss-Hopkins obstruction theory as giving obstructions to constructing a periodic $A$-algebra with prescribed homotopy groups.

\begin{thm}[\ref{cor:if_c_is_complete_we_have_inductive_obstructions_to_existence_of_periodic_algebras}, \ref{rem:obstructions_in_derived_category_in_commutative_case}]
\label{thm:obstructions_to_existence_of_a_periodic_commutative_algebra_in_intro}
Let $S$ be a commutative $\pi_{0} A$-algebra in $\ccat^{\heartsuit}$. Then, there exists a sequence of inductively defined obstructions

\begin{center}
$\theta_{n} \in \Ext_{\Mod_{S}(\ccat^{\heartsuit})}^{n+2, n}(\mathbb{L}^{\ekoperad{\infty}}_{S / \pi_{0}A}, S)$, where $n \geq 1$,
\end{center}
which vanish if and only if there exists a periodic commutative $A$-algebra $R$ such that $\pi_{0} R \simeq S$ as $\pi_{0} A$-algebras. 
\end{thm}
The $\Ext$-groups appearing in the statement are a form of Andr\'{e}-Quillen cohomology; more precisely, it is the Andr\'{e}-Quillen cohomology of $S$ considered as an $\ekoperad{\infty}$-$\pi_{0}A$-algebra. These groups are completely algebraic and relatively computable, giving the statement its power. 

In the body of the paper, we also prove variants of \cref{thm:obstructions_to_existence_of_a_periodic_commutative_algebra_in_intro} which give obstructions to constructing periodic $\mathbf{E}_{k}$-algebras for $k < \infty$, as well as a ``linear'' version relating to periodic $A$-modules, see \cref{thm:obstructions_in_ext_groups_to_lifting_a_potential_stage} and \cref{cor:if_c_is_complete_then_we_have_an_obstruction_theory_to_constructing_a_periodic_a_module}. 

The notions of shift algebra and periodic module are very general, and not necessarily of interest in their own. Rather, the usefulness of \cref{thm:obstructions_to_existence_of_a_periodic_commutative_algebra_in_intro} comes from the fact that for specific choices of the $\infty$-category $\ccat$, periodic $A$-modules can be identified with various kinds of topological objects, yielding a variety of obstruction theories. 

The classical case considered by Goerss and Hopkins is obtained by letting $\ccat$ be the $\infty$-category $\synspectra$ of synthetic spectra, as described in \cite{pstrkagowski2018synthetic}, which informally plays the role of the ``derived $\infty$-category of spectra''. There is an $\infty$-category of synthetic spectra associated to each Adams-type homology theory $E$, and its monoidal unit $\monunit$ is a shift algebra. 

The fundamental property of $\synspectra$ is that its heart $\synspectra^{\heartsuit}$ is equivalent to the category of $E_{*}E$-comodules, while the $\infty$-category $\Mod_{\monunit}^{per}(\synspectra)$ of periodic objects is equivalent to $E$-local spectra. Thus, synthetic spectra act as an intermediary between the world of spectra and the world of comodules. Specialized to $\synspectra$, \cref{thm:obstructions_to_existence_of_a_periodic_commutative_algebra_in_intro} takes the following familiar form. 

\begin{thm}[Goerss-Hopkins, \ref{thm:obstructions_to_realizing_an_algebra_in_comodules}, \cite{moduli_spaces_of_commutative_ring_spectra}]
\label{thm:classical_goerss_hopkins_intro}
Let $E$ be Morava $E$-theory and let $S$ be a commutative algebra in $E_{*}E$-comodules, where $E_{*}E = \pi_{*}(E \wedge E)$. Then, there exists a sequence of inductively defined obstructions

\begin{center}
$\theta_{n} \in \Ext_{\Mod_{S}(\ComodE)}^{n+2, n}(\mathbb{L}^{\ekoperad{\infty}}_{S / E_{*}}, S)$, where $n \geq 1$
\end{center}
which vanish if and only if there exists an $\ekoperad{\infty}$-ring spectrum $E$ such that $E_{*}R \simeq S$ as comodule algebras. 
\end{thm}

In the statement of \cref{thm:classical_goerss_hopkins_intro} we restrict to the case of Morava $E$-theory, since to prove a result at this level of generality one needs to know that the $\infty$-category of synthetic spectra based on E has convergent Postnikov towers. This additional assumption does not appear in the original papers \cite{moduli_spaces_of_commutative_ring_spectra}, \cite{moduli_problems_for_structured_ring_spectra}; we believe that this is a mistake, see \cref{rem:goerss_hopkins_issue_of_completeness}.

To fill the gap, we prove in \cref{thm:synthetic_spectra_based_on_morava_e_theory_complete} that synthetic spectra based on Morava $E$-theory have convergent Postnikov towers, which is related to the fact that the the $E$-local Adams-Novikov spectral sequence always converges \cite[5.3]{hovey1999invertible}. This is enough to yield the Goerss- Hopkins-Miller theorem, whose proof we also review in the interest of being self-contained. 

On the other hand, Postnikov convergence does not hold for an arbitrary Adams-type homology theory and in fact fails already for $H \mathbb{F}_{p}$. This is a consequence of the existence of an $H \mathbb{F}_{p}$-local but not $H \mathbb{F}_{p}$-nilpotent complete spectrum. We couldn't find an example of this phenomena in the literature and so we include one which we learned from Robert Burklund, see \cref{example:non_vanishing_adams_spectral_sequence}. 

In such cases, which are often of interest, variants of \cref{thm:classical_goerss_hopkins_intro} for a fixed comodule can be obtained by adding in an additional step which establishes the needed convergence by hand. This phenomena of additional conditions is familiar from the classical Toda obstruction theory to the existence of spectra realizing a given module $M$ over the Steenrod algebra, as explained in \cite{bhattacharya2016class}, \cite{margolis2011spectra}. To show that these convergence issues can be dealt with, we use Goerss-Hopkins theory to prove a homological version of a classical result of Toda. 

\begin{thm}[Toda, \ref{cor:homological_toda_obstruction_theory}]
\label{thm:toda_obstruction_theory_intro}
Let $H = H \mathbb{F}_{p}$ be the Eilenberg-MacLane spectrum and let $M$ be a $H_{*}H$-comodule which is bounded below. Then, there exists a sequence of inductively defined obstructions

\begin{center}
$\theta_{n} \in \Ext_{\Comod_{H_{*}H}}^{n+2, n}(M, M)$, where $n \geq 1$
\end{center}
which vanish if and only if there exists a spectrum $X$ such that $H_{*}X \simeq M$ as comodules.
\end{thm}
To prove the convergence of Postnikov towers needed to establish \cref{thm:toda_obstruction_theory_intro}, we have to work a little harder than in the case of Morava $E$-theory. The argument we give uses in an essential way the vanishing lines in the $E_{2}$-term of the classical Adams spectral sequence, which are of positive slope, proving the result only in the bounded below case.

In fact, our arguments prove variants of Toda's obstruction theory hold whenever the needed vanishing lines are present. Notably, one also obtains an obstruction theory to realization of bounded below $BP_{*}BP$-comodules, see \cref{rem:toda_obstruction_theory_for_bp}. 

\subsection{Mapping spaces and spectral sequences}

While \cref{thm:obstructions_to_existence_of_a_periodic_commutative_algebra_in_intro} and \cref{thm:classical_goerss_hopkins_intro} allow one to prove that certain realizations exist, one can also say something about the mapping spaces between them. To state this result, it will be convenient for us to sketch our constructions. 

As before, let $\ccat$ be a graded symmetric monoidal prestable $\infty$-category and let $A$ be a commutative shift algebra in $\ccat$, which we recall means that $\pi_{*} A \simeq (\pi_{0}A)[\tau]$. Our goal is to describe the $\infty$-category of periodic $A$-algebras, that is, those that satisfy $\pi_{*} R \simeq \pi_{0} R \otimes _{\pi_{0} A} \pi_{*}A$; the main idea is to interpolate between $A$-algebras and $\pi_{0}A$-algebras by using the Postnikov tower of $A$. 

By an easy computation one sees that an $A$-module $R$ is periodic if and only if $\pi_{0} A \otimes _{A} R$ is discrete. This condition can be naturally generalized, so that we say that a commutative $A_{\leq n}$-algebra $R_{n}$ is a \emph{potential $n$-stage} if $\pi_{0} A \otimes _{\atrun{n}} R_{n}$ is discrete. This leads to the \emph{Goerss-Hopkins tower} of $\infty$-categories

\begin{center}
$\textnormal{CAlg}(\Mod_{A}^{per}) \rightarrow \ldots \rightarrow \textnormal{CAlg}(\mathcal{M}_{1}) \rightarrow \textnormal{CAlg}(\mathcal{M}_{0})$,
\end{center}
where $\textnormal{CAlg}(\mathcal{M}_{n})$ is the $\infty$-category of potential $n$-stages and the connecting functors are given by extension of scalars.

It is clear from the definitions that the base of this tower is the category of discrete $\pi_{0}A$-algebras; moreover, the composite $\textnormal{CAlg}(\Mod_{A}^{per}) \rightarrow \textnormal{CAlg}(\mathcal{M}_{0})$ can be identified with taking $\pi_{0}$. Thus, any discrete $\pi_{0}A$-algebra $S$ determines a point in the base, and to realize $S$ as the homotopy of a periodic algebra is the same as to lift the given point to $\textnormal{CAlg}(\Mod_{A}^{per})$. 

If $\ccat$ has convergent Postnikov towers, the Goerss-Hopkins tower is a limit diagram; to lift the chosen point to the top of the tower it is then enough to construct a compatible sequence of lifts along the functors $\textnormal{CAlg}(\mathcal{M}_{n+1}) \rightarrow \textnormal{CAlg}(\mathcal{M}_{n})$. Since the Postnikov tower of $A$ is a tower of square-zero extensions, one can find cohomological obstructions to the existence of such lifts using the theory of the cotangent complex. These are exactly the obstructions appearing in the statement of \cref{thm:obstructions_to_existence_of_a_periodic_commutative_algebra_in_intro}.

As an another consequence of the convergence of the Goerss-Hopkins tower, we have that if $R, S$ are periodic $A$-algebras, then $\map(R, S) \simeq \varprojlim \map(R_{n}, S_{n})$, where $R_{n} \simeq \atrun{n} \otimes _{A} R$ and $S_{n} \simeq \atrun{n} \otimes _{A} S$ are the corresponding potential $n$-stages. Thus, we can understand $\map(R, S)$ by first describing the spaces $\map(R_{n}, S_{n})$, which can be done inductively. 

It is not hard to see that when $n = 0$, $\map(R_{n}, S_{n})$ can be identified with the set of maps $\pi_{0} R \rightarrow \pi_{0} S$ of $\pi_{0} A$-algebras, considered as a discrete space. It is then natural to ask for the difference between the mapping spaces for $n$ and $(n-1)$. This is provided by the following result.

\begin{thm}[\ref{prop:relation_between_mapping_spaces_between_potential_stages_in_the_multiplicative_case}]
\label{thm:homotopy_groups_of_space_of_lifts_of_a_map_intro}
Let $R_{n}, S_{n}$ be potential $n$-stages and let $\phi: R_{n-1} \rightarrow S_{n-1}$ be a map between the corresponding potential $(n-1)$-stages. Then, there exists an obstruction

\begin{center}
$\theta_{\phi} \in \Ext_{\Mod_{\pi_{0}R_{n}}(\ccat^{\heartsuit})}^{n+1, n}(\mathbb{L}^{\ekoperad{\infty}}_{\pi_{0} R_{n} / \pi_{0}A}, \pi_{0}S_{n})$
\end{center}
which vanishes if and only if $\phi$ lifts to a map $R_{n} \rightarrow S_{n}$. If such a lift exists, then the homotopy groups of the space $F_{\phi}$ of lifts are given by the formula

\begin{center}
$\pi_{k} F_{\phi} \simeq \Ext_{\Mod_{\pi_{0}R_{n}}(\ccat^{\heartsuit})}^{n-k, n}(\mathbb{L}^{\ekoperad{\infty}}_{\pi_{0} R_{n} / \pi_{0}A}, \pi_{0}S_{n})$.
\end{center}
\end{thm}
Note that for simplicity, we gave only the description of the homotopy groups of the space of lifts, but the more precise version on \cref{thm:homotopy_groups_of_space_of_lifts_of_a_map_intro} given in the main body of the text in fact identifies $F_{\phi}$ with a space of paths in a certain Andr\'{e}-Quillen cohomology space. As in the case of \cref{thm:obstructions_to_existence_of_a_periodic_commutative_algebra_in_intro}, we also prove variants for $\mathbf{E}_{k}$-algebras.

Note that if $R, S$ are periodic $A$-algebras, then the expression $\map(R, S) \simeq \varprojlim \map(R_{n}, S_{n})$ of the mapping space between them as a limit of a tower leads in a natural way to a spectral sequence; in this context, \cref{thm:homotopy_groups_of_space_of_lifts_of_a_map_intro} can be interpreted as describing the relevant $E_{1}$-term. Restricting to the case of realizations of commutative ring spectra, we recover the following classical result of Goerss and Hopkins. 

\begin{thm}[\ref{cor:mapping_space_spectral_sequence_ring_spectra}]
\label{thm:intro_spectral_sequence_for_mapping_space_between_einfty_ring_spectra}
Let $E$ be Morava $E$-theory and let $\phi:R \to S$ be a homomorphism of $E$-local $\ekoperad{\infty}$-ring spectra. Then, there is a first quadrant spectral sequence converging to 
\begin{center}
$\pi_{t-s}(\map_{\textnormal{CAlg}(\spectra_{E})}(R, S); \phi)$,
\end{center}
the homotopy groups of the space of $\ekoperad{\infty}$-ring maps from $R$ to $S$ based at $\phi$, with the $E_{1}$-term 

\begin{center}
$E_1^{s,t} = \Ext_{\Mod_{E_*R}(\ComodE)}^{2t-s,s}(\mathbb{L}^{\ekoperad{\infty}}_{E_*R/E_*}, E_*S)$
\end{center}
for $s > 0$ and $E_1^{0,0} = \map_{\CAlg(\ComodE)}(E_*R, E_*S)$.
\end{thm}
The spectral sequence of \cref{thm:intro_spectral_sequence_for_mapping_space_between_einfty_ring_spectra} is a the spectral sequence of a tower of spaces, so that it is not necessary to start with a point $\phi \in \map(R, S)$; rather, such a point can be constructed inductively using the obstructions of \cref{thm:homotopy_groups_of_space_of_lifts_of_a_map_intro}. As before, we state the result for Morava $E$-theory, but it is valid for any Adams-type homology theory after imposing restrictions that guarantee the convergence of the associated Goerss-Hopkins tower. 

\subsection{Related work}
It is important to stress that there is real power in the fact that Goerss-Hopkins theory yields not only information about the existence of a realization, but gives a description of the whole $\infty$-category of such realizations, describing it as a limit of a tower. These towers can be used to prove a variety of results; for example, they were used by Lurie and Hopkins to study the Brauer groups of Morava $E$-theory \cite{lurie_hopkins_brauer_group}. 

To give a different example, Goerss-Hopkins towers were also used by the first author to prove that at large primes, there is an equivalence $h \spectra_{E} \simeq h \dcat(E_{*}E)$ of homotopy categories between $E$-local spectra and differential $E_{*}E$-comodules \cite{pstragowski_chromatic_homotopy_algebraic}. The proof of that result proceeds by comparing topological and algebraic towers and so crucially depends on the additional generality provided in this note.

We should mention that in the original papers \cite{moduli_spaces_of_commutative_ring_spectra, moduli_problems_for_structured_ring_spectra}, Goerss and Hopkins allow one to work with different operads on the topological and algebraic side, spelling out an explicit compatibility condition. This allows, for example, to introduce power operations as part of the realization data. For ease of exposition, we do not pursue this idea here.

A different approach to Goerss-Hopkins obstruction theory is present in the work of Mazel-Gee \cite{mazel2018goerss}. Compared to ours, the approach of Mazel-Gee is closer in spirit to the original works of Goerss and Hopkins, generalizing the theory to a setting of an arbitrary presentable $\infty$-category through the use of model $\infty$-categories. 

\subsection{Notation and conventions}

By an \emph{$\infty$-category} we always mean a quasicategory. We freely use the theory of $\infty$-categories as developed by Joyal and Lurie, the standard reference is \cite{lurie_higher_topos_theory}. All constructions should be understood in the homotopy-invariant sense, in particular limits and colimits. 

If $f: c \rightarrow d$ is a map in a presentable $\infty$-category, then we say that $f$ is an \emph{$n$-equivalence} if it induces an equivalence $\tau_{\leq n} c \rightarrow \tau_{\leq n} d$ between $n$-truncations. If $\ccat$ is prestable, this is equivalent to saying that $f$ induces isomorphisms $\pi_{k} c \rightarrow \pi_{k} d$ between homotopy groups for $k \leq n$. 

If $\ccat$ is a prestable $\infty$-category and $a, b \in \ccat$, we write $\Ext^{s}_{\ccat}(a, b) := \pi_{0} \map(a, \Sigma^{s} b)$. If $\acat$ is an abelian category and $a, b \in \acat$, then we also write $\Ext^{s}_{\acat}(a, b) := \Ext^{s}_{\dcat(\acat)}(a, b)$, which reduces to the $\Ext$-groups of $\acat$ in the usual sense.

If $A$ is an $\ekoperad{k}$-algebra object in some symmetric monoidal $\infty$-category, where $1 \le k \le \infty$, then there is an $\infty$-category of $\ekoperad{k}$-$A$-modules $\Mod_A^{\ekoperad{k}}$ \cite{higher_algebra}[3.3.3]. If $k = 1$, this is equivalent to the $\infty$-category of $A$-bimodules, and when $k = \infty$, then this is the $\infty$-category of left $A$-modules.  We write
\[	\map^{\ekoperad{k}}_A(M,N)	\]
for the space of $\ekoperad{k}$-$A$-module maps $M \to N$. If $f:A \to B$ is a map of $\ekoperad{k}$-algebra objects in some symmetric monoidal $\infty$-category, then it induces an adjunction
\[	f^*:\Mod_A^{\ekoperad{k}} \leftrightarrows \Mod_B^{\ekoperad{k}} :f_*	\]
where $f^*$ is the base change functor and $f_*$ is the forgetful functor.

\subsection{Acknowledgements}

Considering the subject matter, and the fact that he supervised both of the authors, it comes as no surprise that we are heavily indebted to Paul Goerss, whom we would like to thank for his support and guidance. We would like to thank Robert Burklund for sharing his example of a non-nilpotent complete spectrum.

The first author would also like to thank Kyoto University, where the interesting parts of this note were first written down. The second author was inspired to work on this project by the Talbot workshop on obstruction theory, and would particularly like to thank Maria Basterra and Sarah Whitehouse, the mentors of the workshop, as well as Dominic Culver, for helpful conversations there.

\section{Periodic modules in prestable \texorpdfstring{$\infty$}{infinity}-categories} 

In this section we review the theory of Grothendieck prestable $\infty$-categories, due to Lurie. Then, we introduce the notion of a periodic module over a shift algebra, and we derive their basic properties. 

Informally, the kind of $\infty$-categories we consider come equipped with a periodicity operator $\tau$. Then a shift algebra $A$ is an associative algebra whose homotopy groups form a polynomial algebra $\pi_{*} A \simeq (\pi_{0} A)[\tau]$, and an $A$-module $M$ is periodic if $\pi_{*} M \simeq \pi_{*} A \otimes _{\pi_{0} A} \pi_{0} M$. In this note, we phrase Goerss-Hopkins theory as giving obstructions to constructing a periodic algebra with specified homotopy groups.

As explained in the introduction, the notions of a shift algebra and periodicity module are very general, and not necessarily of interest in their own right. Rather, working at this level of generality is useful since one can construct various ``designer'' prestable $\infty$-categories in which periodic modules can be identified with a desired type of topological objects, see \cref{example:periodic_modules_in_synthetic_spectra_same_as_elocal_spectra}, \cref{example:periodic_modules_in_synthetic_emodules_same_as_kn_local_emodules}. 

\begin{defin}\label{defin:grothendieck_prestable}
Recall that an $\infty$-category $\ccat$ is \emph{Grothendieck prestable} if there exists a presentable, stable $\infty$-category $\dcat$ equipped with a $t$-structure ($\dcat_{\geq 0}, \dcat_{\leq 0})$ compatible with filtered colimits and an equivalence $\ccat \simeq \dcat_{\geq 0}$ \cite[C.1.4.2]{lurie_spectral_algebraic_geometry}. This in particular implies that $\ccat$ is presentable and additive.
\end{defin}

\begin{example}
If $R$ is a connective ring spectrum, then the $\infty$-category of $\Mod_{R}(\spectra_{\geq 0})$ of connective $R$-modules is Grothendieck prestable. In fact, this example is universal in the sense that any other Grothendieck prestable $\infty$-category can be obtained by a left exact localization of one of this form \cite[C.2.4.1]{lurie_spectral_algebraic_geometry}. 
\end{example}

\begin{rem}
\label{rem:grothendieck_prestable_embeds_into_stable}
If $\ccat$ is Grothendieck prestable, there is a canonical choice of a stable $\dcat$ equipped with a $t$-structure such that $\ccat \simeq \dcat_{\geq 0}$. Namely, one can take $\dcat$ to be $Sp(\ccat)$, the $\infty$-category of spectrum objects in $\ccat$ \cite[C.1.2.10, C.3.1.5]{lurie_spectral_algebraic_geometry}. 

The $\infty$-category $Sp(\ccat)$ is called the \emph{stabilization} of $\ccat$, and it is the universal stable, presentable $\infty$-category equipped with a cocontinuous functor out of $\ccat$. One can show that the functor $\ccat \rightarrow Sp(\ccat)$ is fully faithful and its essential image is the connective part of a canonical $t$-structure compatible with filtered colimits. Thus, Grothendieck prestable $\infty$-categories are a convenient way to encode a stable $\infty$-category \emph{together} with a choice of a $t$-structure. 
\end{rem}
If $\ccat$ is Grothendieck prestable, then by $\ccat^{\heartsuit}$ we denote the \emph{heart}; that is, the subcategory of discrete objects. One can show that the category $\ccat^{\heartsuit}$ is Grothendieck abelian \cite[1.3.5.23]{higher_algebra}. Note that the embedding of \cref{rem:grothendieck_prestable_embeds_into_stable} induces an equivalence $\ccat^{\heartsuit} \simeq Sp(\ccat)^{\heartsuit}$, where by the latter we mean the heart of the canonical $t$-structure on $\spectra(\ccat)$, justifying the terminology. 

\begin{defin}
Let $\ccat$ be Grothendieck prestable and let $X \in \ccat$. Then, by the \emph{homotopy groups} $\pi_{i} X \in \ccat^{\heartsuit}$ we mean the homotopy groups of the image of $X$ in $Sp(\ccat)$ with respect to its canonical $t$-structure.
\end{defin}
One can give an explicit formula for the homotopy groups. Namely, for $i \geq 0$ we have $\pi_{i} X \simeq (\Omega^{i} X)_{\leq 0}$ as objects of $\ccat^{\heartsuit}$, where by $(-)_{\leq 0}$ we denote the $0$-truncation, while the negative homotopy groups vanish. In particular, $\pi_{0} X \simeq X_{\leq 0}$ is just the discretization of $X$. 

\begin{defin}
\label{defin:separated_and_complete_prestable_infty_cats}
We say a Grothendieck prestable $\infty$-category $\ccat$ is \emph{separated} if the homotopy groups detect equivalences; in other words, if a map $X \to Y$ is an equivalence in $\ccat$ if and only if $\pi_nX \to \pi_nY$ is an isomorphism in $\ccat^\heartsuit$ for every $n$.

We say $\ccat$ is \emph{complete} if the Postnikov towers in $\ccat$ converge, meaning that for every $X \in \ccat$, the natural map $X \to \varprojlim_n X_{\le n}$ is an equivalence.
\end{defin}
Note that a complete Grothendieck prestable $\infty$-category is separated, but the converse is not true. 

\begin{example}
If $\acat$ is a Grothendieck abelian category, then the connective derived $\infty$-category $D(\acat)_{\geq 0}$ is separated Grothendieck prestable and we have $D(\acat)_{\geq 0}^{\heartsuit} \simeq \acat$. In fact, one can show that it is a universal such $\infty$-category \cite[C.5.4.9]{lurie_spectral_algebraic_geometry}. 

On the other hand, the $\infty$-category $D(\acat)_{\geq 0}$ is not necessarily complete, the counterexample due to Neeman being the case of the category $\acat$ of representations of the additive group in positive characteristic \cite{neeman2011non}.
\end{example}

\begin{example}
Any stable, presentable $\infty$-category is Grothendieck prestable, but it is never separated in the sense of \cref{defin:separated_and_complete_prestable_infty_cats} unless it is zero. 
\end{example}

\begin{defin}\label{defin:grading}
A \emph{grading} on a symmetric monoidal Grothendieck prestable $\infty$-category $\ccat$ is a choice of a distinguished autoequivalence which we denote by $c \mapsto c[1]$, together with a natural equivalence $c[1] \otimes d \simeq (c \otimes d)[1]$. 
\end{defin}

As $c \mapsto c[1]$ is an autoequivalence of $\ccat$, iterating it defines autoequivalences $c \mapsto c[k]$, for $k \in \mathbb{Z}$. We then have $c[l] \otimes d[k] \simeq (c \otimes d)[k+l]$ for $c, d \in \ccat$ and $k, l \in \mathbb{Z}$. Note that we implicitly assume that the given symmetric monoidal structure is presentable, that is, that it commutes with colimits in each variable. We will denote the unit of $\ccat$ by $\monunit$. 

\begin{rem}
\label{rem:grading_on_heart}
Since a grading on $\ccat$ is an auto-equivalence, it preserves the subcategory $\ccat^{\heartsuit}$ of discrete objects, and thus makes $\ccat^{\heartsuit}$ into a graded abelian category (in the sense analogous to \cref{defin:grading}). This grading is compatible with the one on $\ccat$, in the sense that
\[	\pi_i A[k] \simeq (\pi_i A)[k].	\]
Moreover, there are bigraded Ext groups between objects in $\ccat^{\heartsuit}$, defined by
\[	\Ext^{s,t}_{\ccat^{\heartsuit}}(A,B) = \mathrm{map}_{D(\ccat^{\heartsuit})}(A, \Sigma^s B[-t]).	\]
\end{rem}

Similarly, a symmetric monoidal structure on $\ccat$ induces one on $\ccat^{\heartsuit}$, the unique one for which the truncation functor $\pi_{0}: \ccat \rightarrow \ccat^{\heartsuit}$ is symmetric monoidal. In particular, the unit of $\ccat^{\heartsuit}$ is given by $\pi_{0} \monunit \simeq \monunit_{\leq 0}$. 

\begin{example}
Any symmetric monoidal Grothendieck prestable $\infty$-category can be trivially graded by making the chosen autoequivalence the identity.
\end{example}

\begin{example}
If $\acat$ is the abelian category of graded modules over a graded ring, then $D(\acat)_{\geq 0}$ is canonically a graded symmetric monoidal, separated Grothendieck prestable $\infty$-category. Here, the symmetric monoidal structure is induced from the tensor product, and the grading is induced by the internal shift of modules.
\end{example} 

\begin{example}
\label{example:synspectra_is_a_graded_sym_mon_separated_grothendieck_prestable_category}
The $\infty$-category $\synspectra$ of hypercomplete, connective synthetic spectra, as defined in \cite{pstrkagowski2018synthetic}, is a graded symmetric monoidal, separated Grothendieck prestable $\infty$-category. As explained in the introduction, this $\infty$-category is the context for the Goerss-Hopkins obstruction theory to realizing an algebra in comodules as homology of a ring spectrum. This $\infty$-category will be discussed in more detail in section 6.\end{example}

We will now describe a theory of shift algebras and periodic modules over them. Throughout the rest of the section, we assume that $\ccat$ is a fixed graded symmetric monoidal, separated Grothendieck prestable $\infty$-category. 

\begin{defin}
\label{defin:shift_algebra}
A \emph{shift algebra} $A$ is an associative algebra $A \in \textnormal{Alg}(\ccat)$ equipped with a map $\tau: \Sigma A[-1] \rightarrow A$ of right $A$-modules which induces an isomorphism $\pi_{*} A \simeq \pi_{0} A[\tau] $, where the latter is the graded algebra in $\ccat^{\heartsuit}$ given by $(\pi_{0} A[\tau])_{k} \simeq (\pi_{0} A)[-k]$. 
\end{defin}
Our notion of a shift algebra is rather weak, as we only require that the homotopy groups of $A$ resemble a polynomial algebra over $\pi_{0}A$. Note that $A$ itself need not be a $\pi_{0}A$-algebra.
\begin{example}
\label{example:unit_of_synspectra_a_periodicity_algebra}
In the $\infty$-category $\synspectra$ of synthetic spectra of \cref{example:synspectra_is_a_graded_sym_mon_separated_grothendieck_prestable_category}, the monoidal unit $\monunit$ is given by the synthetic analogue $\nu S^{0}_{E}$ of the $E$-local sphere, and is in fact a shift algebra \cite[4.61, 4.21]{pstrkagowski2018synthetic}. This is an example of a shift algebra which is commutative. 
\end{example}

\begin{example}
\label{example:over_a_discrete_ring_a_unique_shift_algebra}
If $k \in \textnormal{CAlg}(\ccat^{\heartsuit})$ is a discrete commutative algebra, then there is a unique up to equivalence shift $k$-algebra $A$ with $\pi_{0} A \simeq k$. To see that one exists, notice that $F_{k}(\Sigma k[-1])$, the free associative $k$-algebra on $\Sigma k[-1]$, is a shift $k$-algebra. 

To see that this is the only one, assume that $A$ is any other shift $k$-algebra with $\pi_{0} A \simeq k$. Then, the $k$-linear composite $\Sigma k[-1] \rightarrow \Sigma A[-1] \rightarrow A$ induces a map $F_{k}(\Sigma k[-1]) \rightarrow A$ which is necessarily an equivalence, as we see by inspecting the homotopy groups. 
\end{example}

If $M$ is a left $A$-module, then we have a morphism $\tau \otimes _{A} M: \Sigma M [-1] \rightarrow M$, which by abuse of notation we will also denote by $\tau$. Note that this is in general only a morphism of the underlying objects, but not necessarily of $A$-modules, unless $A$ is commutative. 

We will now introduce a notion of a periodic module over a shift algebra. Roughly, a periodic $A$-module is one on which $\tau$ acts as close to an equivalence as possible. 

\begin{prop}
\label{prop:different_characterizations_of_periodic_objects}
Let $A$ be a shift algebra, and let $M$ be a left $A$-module. Then the following conditions are equivalent:

\begin{enumerate}
\item $\pi_{*} M \simeq \pi_{*} A \otimes _{\pi_{0} A} \pi_{0} M$.
\item $\pi_{0} A \otimes _{A} M$ is a discrete object of $\ccat$.
\item $\tau: \Sigma M[-1] \rightarrow M$ is a $1$-connective cover. 
\end{enumerate}
\end{prop}

\begin{proof}
By looking at the homotopy groups, we see that $\tau$ induces a cofibre sequence 
\begin{center}
$\Sigma A[-1] \rightarrow A \rightarrow \pi_{0}A$
\end{center}
of right $A$-modules. By tensoring it with $M$ we obtain a cofibre sequence
\begin{equation}\label{eq:periodic_cofibre_sequence}
\Sigma M[-1] \rightarrow M \rightarrow \pi_{0}A \otimes _{A} M
\end{equation}
which immediately implies that $(2)$ and $(3)$ are equivalent. 

If $(2)$ and $(3)$ are satisfied, then the long exact sequence of homotopy groups of \eqref{eq:periodic_cofibre_sequence} gives
\begin{align*}
	\pi_0M &\simeq \pi_0A \otimes_A M, \\
	\pi_kM &\simeq \tau \pi_{k-1}M[-1] = \pi_1A \otimes \pi_{k-1}M \quad \text{for }k>0.
\end{align*}
This gives $(1)$. Conversely, if $(1)$ is satisfied, then $\pi_kM \simeq \tau^k\pi_0M[-k]$, which immediately implies $(3)$.
\end{proof}

\begin{defin}
\label{defin:periodic_objects_in_a_grothendieck_prestable_infty_category}
We say that a left $A$-module $M \in \Mod_{A}(\ccat)$ is \emph{periodic} if it satisfies the equivalent conditions of \cref{prop:different_characterizations_of_periodic_objects}. We denote the full subcategory of $\Mod_A(\ccat)$ spanned by the periodic modules by $\Mod_{A}^{per}(\ccat)$. 
\end{defin}

For specific choices of $\ccat$, $\Mod_{A}^{per}(\ccat)$ can be often identified with various $\infty$-categories of interest. Let us give a couple of examples.

\begin{example}
\label{example:periodic_modules_in_synthetic_spectra_same_as_elocal_spectra}
If $\ccat$ is the $\infty$-category of synthetic spectra considered in \cref{example:unit_of_synspectra_a_periodicity_algebra}, the $\infty$-category $\Mod_{\monunit}(\synspectra)$ of periodic modules can be shown to be equivalent to the $\infty$-category $\spectra_{E}$ of $E$-local spectra \cite[4.36, 5.6]{pstrkagowski2018synthetic}.
\end{example}

\begin{example}
\label{example:periodic_modules_in_synthetic_emodules_same_as_kn_local_emodules}
In their paper \cite{lurie_hopkins_brauer_group}, Hopkins and Lurie introduce a graded symmetric monoidal, separated Grothendieck prestable $\infty$-category $\synspectra_{E}$ of \emph{synthetic $E$-modules}, where $E$ is the Morava $E$-theory ring spectrum. They show that the monoidal unit is a shift algebra, and that periodic objects can be identified with the $\infty$-category of $K(n)$-local $E$-modules in spectra.
\end{example}

\begin{example}
If $R$ is a discrete commutative ring, then there is a unique shift algebra $A$ in $\dcat(R)_{\geq 0}$ with $\pi_{0} A \simeq R$, see \cref{example:over_a_discrete_ring_a_unique_shift_algebra}. The $\infty$-category of periodic $A$-modules can be identified with the \emph{1-periodic derived $\infty$-category} of $R$ (cf. \cite{stai}, \cite{barnes2011monoidality}) -- that is, the derived $\infty$-category of $R$-modules $M$ equipped with an endomorphism $d:M \to M$ such that $d^2 = 0$.
\end{example}

One can notice that in all of our examples, the $\infty$-category $\Mod_{A}^{per}(\ccat)$ of periodic $A$-modules is a stable $\infty$-category. This is not an accident: under very minor assumptions, one can show that a periodic $A$-module uniquely encodes a $\tau$-local one; that is, an $A$-module $M$ such that $\tau: \Sigma M[-1] \rightarrow M$ is an equivalence. One sees easily that no non-zero module in $\ccat$ could satisfy this condition, so that to make sense of this we need to allow more general modules in $Sp(\ccat)$.

For simplicity, we will assume that $A$ is an $\mathbf{E}_{2}$-algebra, so that the tensor product endows the $\infty$-category of $A$-modules with a monoidal structure and thus $\tau: \Sigma M[-1] \rightarrow M$ is canonically a map of $A$-modules for any $M \in \Mod_{A}(\ccat)$. One can almost certainly get away with weaker assumptions: see for example the treatment by Lurie of localizations of modules over ring spectra \cite[7.2.4.17]{higher_algebra}.

\begin{defin}
\label{defin:tau_local_module}
Let $A$ be a shift $\mathbf{E}_{2}$-algebra. We say an $A$-module $M \in \Mod_{A}(\spectra(\ccat))$ is \emph{$\tau$-local} if $\tau: \Sigma M[-1] \rightarrow M$ is an equivalence. We denote the $\infty$-category of $\tau$-local $A$-modules by $\Mod^{\tau^{-1}}_{A}(\spectra(\ccat))$.
\end{defin}

The inclusion $\Mod_{A}^{\tau^{-1}}(\spectra(\ccat)) \hookrightarrow \Mod_{A}(\spectra(\ccat))$ is easily seen to have a left adjoint $L_{\tau}$ given by the explicit formula 

\begin{center}
$L_{\tau}M := \varinjlim \ M \rightarrow \Sigma^{-1}M[1] \rightarrow \Sigma^{-2} M[2] \rightarrow \ldots $.
\end{center}
Thus, $\Mod_{A}^{\tau^{-1}}(\spectra(\ccat))$ is a localization of $\Mod_{A}(\spectra(\ccat))$. This localization is smashing, since $\tau$-local modules are clearly closed under colimits, and compatible with the monoidal structure, since the tensor product $M \otimes _{A} N$ is $\tau$-local if one of $M$, $N$ is. 

\begin{prop}
\label{prop:periodic_a_modules_equivalent_to_tau_local_ones}
The connective cover functor $(-)_{\geq 0}$ and the localization $L_{\tau}$ restrict to inverse equivalences $\Mod_{A}^{per}(\ccat) \simeq \Mod_{A}^{\tau^{-1}}(\spectra(\ccat))$ between the $\infty$-categories of periodic $A$-modules in $\ccat$ and $\tau$-local $A$-modules in $\spectra(\ccat)$. 
\end{prop}

\begin{proof}
If $M$ is periodic, then the map $\tau: M \rightarrow \Sigma^{-1} M[1]$ is a $0$-connective cover. Since the $t$-structure on $\spectra(\ccat)$ is compatible with filtered colimits, we deduce that the same is true for the inclusion $M \rightarrow LM \simeq \varinjlim \Sigma^{-k} M[k]$ and so the canonical map $M \rightarrow (LM)_{\geq 0}$ is an equivalence. 

We deduce that it is enough to show that the functor $(-)_{\geq 0}: \Mod^{\tau^{-1}}_{A}(\spectra(\ccat)) \rightarrow \Mod_{A}(\ccat)$ is conservative. This is clear, since a map $M \rightarrow N$ is an equivalence if and only if $\pi_{*} M \rightarrow \pi_{*} N$ is an isomorphism. Since the latter are modules over $(\pi_{0}A)[\tau^{\pm 1}]$ when $M, N$ are $\tau$-local, this happens if and only if $\pi_{0} M \rightarrow \pi_{0}N$ is an isomorphism, proving the claim.
\end{proof}

\begin{cor}
\label{cor:infinity_cat_of_periodic_modules_is_stable}
Let $A$ be a shift $\ekoperad{2}$-algebra. Then, the $\infty$-category $\Mod_{A}^{per}(\ccat)$ of periodic $A$-modules is stable. 
\end{cor}

\begin{proof}
This is immediate from \cref{prop:periodic_a_modules_equivalent_to_tau_local_ones}, since the $\infty$-category $\Mod_{A}^{\tau^{-1}}(\spectra(\ccat))$ is clearly stable, in fact a thick subcategory of $\Mod_{A}(\spectra(\ccat))$.
\end{proof}

\section{Modules over truncations}

In this section, we discuss the relations between the $\infty$-categories of modules over Postnikov truncations of an associative algebra $A$. We make no claim to originality, as we closely follow the account of Hopkins and Lurie \cite[7.3]{lurie_hopkins_brauer_group}, filling in the details and making sure that the arguments work in the generality in which we need them. 

Throughout the section, we fix a symmetric monoidal, separated Grothendieck prestable $\infty$-category $\ccat$ and we let $A \in \textnormal{Alg}(\ccat)$ be an associative algebra. To fix ideas, by a module we will mean a left module in $\ccat$, and we use the shorthand notation $\Mod_{A} := \Mod_{A}(\ccat)$. 

As $\ccat$ is the connective part of a stable $\infty$-category and a compatible symmetric monoidal structure, the truncation functors $\tau_{\le n}$ are compatible with the symmetric monoidal structure \cite[2.2.1.10]{higher_algebra}, and so each $\tau_{\le n}:\ccat \to \ccat$ is lax symmetric monoidal. So we can associate to $A$ a tower of associative algebras

\begin{equation}\label{eq:A_truncation_tower}
A \rightarrow \ldots \rightarrow \atrun{n} \rightarrow \ldots \rightarrow \atrun{1} \rightarrow \atrun{0}.
\end{equation}
given by the Postnikov truncations. This is in fact necessarily a tower of square-zero extensions \cite[7.4.1.28]{higher_algebra}, a fact which will become important later. This tower of algebras induces a tower of Grothendieck prestable $\infty$-categories 

\begin{equation}\label{eq:Mod_A_truncation_tower}
\Mod_{A} \rightarrow \ldots \rightarrow \Mod_{\atrun{n}} \rightarrow \ldots \rightarrow \Mod_{\atrun{1}} \rightarrow \Mod _{\atrun{0}},
\end{equation}
where the indicated functors are given by base change along the maps in \eqref{eq:A_truncation_tower}.

In general, we would like to use the tower \eqref{eq:Mod_A_truncation_tower} to inductively describe $\Mod_{A}$. To make this work, one has to understand three things, namely
\begin{enumerate}
\item the $\infty$-category $\Mod_{\atrun{0}}$, 
\item the difference between $\Mod_{\atrun{n}}$ and $\Mod_{\atrun{n-1}}$, and 
\item the relation between $\Mod_{A}$ and $\varprojlim \Mod_{A_{\leq n}}$. 
\end{enumerate}

In this section, we will give answers to $(2)$ and $(3)$ in the generality of an arbitrary associative algebra $A$, and an answer to $(1)$ that applies to the kind of shift algebras we have in mind. We will proceed in an order reverse to the one given above. 

We start with the relation between $\Mod_{A}$ and $\Mod_{\atrun{n}}$, which, assuming that the ambient Grothendieck prestable $\infty$-category $\ccat$ is complete, is as one would expect.

\begin{lemma}
\label{lemma:any_k_truncated_module_canonically_a_module_over_k_truncation}
The restriction of scalars functor $\Mod_{\atrun{k}} \rightarrow \Mod_{A}$ induces an equivalence between the $\infty$-categories of $k$-truncated objects. In other words, any $k$-truncated $A$-module is uniquely an $\atrun{k}$-module. 
\end{lemma}

\begin{proof}
Since the truncation endofunctor $(-)_{\leq k}: \ccat \rightarrow \ccat$ is lax symmetric monoidal, it induces a functor $(-)_{\leq k}: \Mod_{A} \rightarrow \Mod_{\atrun{k}}$. This functor is readily verified to be inverse to the restriction of scalars when restricted to the subcategories of $k$-truncated objects. 
\end{proof}

\begin{prop}
\label{prop:if_ambient_grothendieck_infty_cat_prestable_module_category_a_limit_of_modules_over_truncation}
Assume that $\ccat$ is complete; in other words, that Postnikov towers in $\ccat$ converge. Then, the extension of scalars functors $\Mod_{A} \rightarrow \Mod_{\atrun{k}}$ induce an equivalence $\Mod_{A} \simeq \varprojlim \Mod_{\atrun{n}}$. 
\end{prop}

\begin{proof}
Notice that since $\ccat$ is complete, the same is true for the $\infty$-category $\Mod_{B}$ of modules over any associative algebra $B \in \textnormal{Alg}(\ccat)$, so that in any such case we have an equivalence $\Mod_{B} \simeq \varprojlim \tau_{\leq k} \Mod_{B}$, where by $\tau_{\leq k}\Mod_{B}$ we denote the subcategory of $k$-truncated objects. 

Using the above statement for $B = \atrun{n}$ we see that we only need to prove that 

\begin{center}
$\Mod_{A} \simeq \varprojlim_{n} \Mod_{\atrun{n}} \simeq \varprojlim_{n} \varprojlim_{k} \tau_{\leq k} \Mod_{\atrun{n}} \simeq \varprojlim_{k} \varprojlim_{n} \tau_{\leq k} \Mod_{\atrun{n}}$, 
\end{center}
where we've used that limits can always be commuted with other limits. By \cref{lemma:any_k_truncated_module_canonically_a_module_over_k_truncation}, the maps $\tau_{\leq k} \Mod_{A} \rightarrow \tau_{\leq k} \Mod_{\atrun{n}}$ are equivalences for $n \geq k$, thus $\varprojlim _{n} \tau_{\leq k} \Mod_{\atrun{n}} \simeq \tau_{\leq k} \Mod_{A}$ for any $k$. We deduce that the needed statement is equivalent to $\Mod_{A} \simeq \varprojlim_{k} \tau_{\leq k} \Mod_{A}$ which is clear from completeness of $\Mod_{A}$. 
\end{proof}

We will now describe the relation between $\Mod_{\atrun{n}}$ and $\Mod_{\atrun{n-1}}$, which can be derived as a formal consequence of the fact that the Postnikov tower of $A$ is a tower of square-zero extensions of a particular form. Our account follows Hopkins and Lurie very closely, especially their introduction of what we call the $\Theta$-functor. 

\begin{lemma}
\label{lemma:extension_along_0_equivalence_is_conservative}
Let $f: A \rightarrow B$ be a map of algebras which is a $0$-equivalence, meaning that $f$ induces an isomorphism $\pi_{0} A \stackrel{\sim}{\to} \pi_{0} B$. Then, the extension of scalars functor $f^{*}: \Mod_{A} \rightarrow \Mod_{B}$ is conservative; that is, it reflects equivalences. 
\end{lemma}

\begin{proof}
Since both $\infty$-categories in question are prestable, a map $M \rightarrow N$ is an equivalence in either one if and only if the cofibre vanishes. Since extension of scalars is cocontinuous, it is enough to show that if $M \in \Mod_{A}$ and $B \otimes _{A} M = 0$, then $M = 0$. Because we assume $\ccat$ to be separated, it is enough to show that $M$ is $n$-connected for all $n \geq 0$, which we will prove by induction. 

To cover the base case, consider the map $M \rightarrow B \otimes_{A} M$. This map is the colimit of a map between bar constructions,
\[  A \otimes A^{\otimes \bullet} \otimes M \to B \otimes A^{\otimes \bullet} \otimes M.    \]
Since $A \to B$ is a $0$-equivalence and everything is $(-1)$-connected, this map is a levelwise $0$-equivalence. Because $0$-equivalences are closed under colimits, we find that $M \to B \otimes_A M$ is a $0$-equivalence. It follows that $M$ is $0$-connected, since $B \otimes _{A} M$ vanishes by assumption. 

Now let $n > 0$ and assume that we know that any $A$-module $M$ such that $B \otimes _{A} M = 0$ is $(n-1)$-connected. Suppose that $M$ is such an $A$-module. Since $M$ is $0$-connected, by prestability we can write $M = \Sigma N$ for some other $A$-module $N$. We have 

\begin{center}
$\Sigma(B \otimes_{A} N) \simeq B \otimes_{A} (\Sigma N) \simeq B \otimes _{A} M = 0$
\end{center} 
and by prestability we deduce that $B \otimes _{A} N = 0$. By the inductive hypothesis, we deduce that $N$ is $(n-1)$-connected and thus $M$ is $n$-connected, which ends the argument.
\end{proof}

\begin{cor}
For any algebra $A$ and $0 \leq l \leq k \leq \infty$, the extension of scalars functor $\atrun{l} \otimes _{\atrun{k}} -: \Mod_{\atrun{k}} \rightarrow \Mod_{\atrun{l}}$ is conservative.
\end{cor}

\begin{proof}
This is immediate from \cref{lemma:extension_along_0_equivalence_is_conservative}. 
\end{proof}

Recall that the Postnikov tower of $A$ is a tower of square-zero extensions of associative algebras \cite[7.4.1.28]{higher_algebra}. Define $F_n \in \Mod_{\atrun{n}}$ by the fibre sequences
\begin{center}
$F_{n} \rightarrow \atrun{n} \rightarrow \atrun{n-1}$,
\end{center}
so that the homotopy of $F_{n}$ is necessarily concentrated in degree $n$. By \cite[7.4.1.26]{higher_algebra}, there is a pullback square in $\textnormal{Alg}(\ccat)$ of the form:
\begin{center}
	\begin{tikzpicture}
		\node (TL) at (0, 1.4) {$ \atrun{n} $};
		\node (TR) at (2.7, 1.4) {$ \atrun{n-1} $};
		\node (BL) at (0, 0) {$ \atrun{n-1} $};
		\node (BR) at (2.7, 0) {$ \atrun{n-1} \oplus \Sigma F_{n}. $};
		\node(S) at (5, 0.7) {$ (\spadesuit) $};
		
		\draw [->] (TL) to (TR);
		\draw [->] (TL) to (BL);
		\draw [->] (TR) to node[right] {$ d $} (BR);
		\draw [->] (BL) to node[below] {$ d_{0} $} (BR);
	\end{tikzpicture}
\end{center}

Here, $\atrun{n-1} \oplus \Sigma F_n$ is a trivial square-zero extension, and $d_{0}$ is the trivial section, while the map $d:\atrun{n-1} \to \atrun{n-1} \oplus \Sigma F_n$ is some other map of algebras, which depends on the actual square-zero extension $\atrun{n} \to \atrun{n-1}$. We now show that this diagram also induces a pullback square of module $\infty$-categories. 

\begin{prop}
\label{prop:pullback_square_of_module_infinity_categories}
The commutative square of module $\infty$-categories and extension of scalars functors 

\begin{center}
	\begin{tikzpicture}
		\node (TL) at (0, 1.3) {$ \Mod_{\atrun{n}} $};
		\node (TR) at (3.4, 1.3) {$ \Mod_{\atrun{n-1}} $};
		\node (BL) at (0, 0) {$ \Mod_{\atrun{n-1}}$};
		\node (BR) at (3.4, 0) {$ \Mod_{\atrun{n-1} \oplus \Sigma F_{n}} $};
		
		\draw [->] (TL) to (TR);
		\draw [->] (TL) to (BL);
		\draw [->] (TR) to node[right] {$ d^{*} $} (BR);
		\draw [->] (BL) to node[below] {$ d_{0}^* $} (BR);
	\end{tikzpicture}
\end{center}
induced by the diagram $(\spadesuit)$ is a pullback square of $\infty$-categories. That is, there is a canonical equivalence between $\Mod_{\atrun{n}}$ and the $\infty$-category of triples $(M, N, \alpha)$, where $M, N \in \Mod_{\atrun{n-1}}$ and $\alpha: d_{0}^* X \simeq d^{*} Y$. 
\end{prop}

\begin{proof}
This appears in the work of Hopkins and Lurie as \cite[7.3.6]{lurie_hopkins_brauer_group}, but for convenience of the reader we recall the argument. There is an adjunction
\[	F:\Mod_{\atrun{n}} \leftrightarrows \Mod_{\atrun{n-1}} \times _{\Mod_{\atrun{n-1} \oplus \Sigma F_{n}}} \Mod_{\atrun{n-1}}: G,	\]
where $F$ is induced by base change around the diagram $(\spadesuit)$, and its right adjoint $G$ is given informally by the formula 
\[	(M, N, \alpha) \mapsto M \times _{d_{0}^* M} N.	\]

We would like to show that $F$ and $G$ are inverse equivalences. We first show that the unit of the adjunction is an equivalence; that is, for any $M \in \Mod_{\atrun{n}}$, the square
\[	\xymatrix{ M \ar[r] \ar[d] & \atrun{n-1} \otimes_{\atrun{n}} M \ar[d] \\ \atrun{n-1} \otimes_{\atrun{n}} M \ar[r] & (\atrun{n-1} \oplus \Sigma F_n) \otimes_{\atrun{n}} M }	\]
is a pullback square. In fact, $(\spadesuit)$ is easily seen to be a pushout square of $A_{\leq n}$-bimodules, so it is taken to a pushout square by tensoring with $M$. But $\ccat$ is prestable, so $\Mod_{A\le n}$ is prestable, which implies that this square is also a pullback.

It is now enough to prove that $G$ is conservative. By a direct calculation, $G$ preserves cofibres and thus by prestability it is enough to check that if $(M, N, \alpha)$ is a triple such that $M \times _{d_{0}^* M} N$ vanishes, then $M = 0$ and $N = 0$. 

Observe that since both $d, d_{0}$ are sections of the projection $\atrun{n-1} \oplus \Sigma F_{n} \rightarrow \atrun{n-1}$, which is an $n$-equivalence, the maps $M \rightarrow d_{0}^* M$ and $N \rightarrow d^{*} N \simeq d_{0}^* M$ are also $n$-equivalences. It follows that the projections $\pi_{0} M \times _{\pi_{0} d_{0}^* M} \pi_{0} N \rightarrow \pi_{0} M$ and $\pi_{0} M \times _{\pi_{0} d_{0}^* M} \pi_{0 } N \rightarrow \pi_{0} N$ are isomorphisms. Since $\pi_{0} (M \times _{d_{0}^* M} N) \rightarrow \pi_{0} M \times _{\pi_{0} d_{0}^* M} \pi_{0} N$ is a surjection by the long exact sequence of homotopy and the source vanishes by assumption, we deduce that $\pi_{0} M = 0$ and $\pi_{0} N = 0$.

We conclude that if $M \times _{d_{0}^* M} N$ vanishes, then $M, N$ are $0$-connected. Arguing inductively as in the proof of \cref{lemma:extension_along_0_equivalence_is_conservative}, we deduce that they are $n$-connected for all $n \geq 0$, and so vanish. This ends the argument. 
\end{proof}

Observe that \cref{prop:pullback_square_of_module_infinity_categories} gives a description of modules over $\atrun{n}$ in terms of modules over $\atrun{n-1}$, which was our goal. Our next step is to show that this description can be further simplified. 

\begin{defin}
\label{defin:hopkins_lurie_theta_functor}
We define the functor $\Theta: \Mod_{\atrun{n-1}} \rightarrow \Mod_{\atrun{n-1}}$ by $\Theta M = d_{*} d_{0}^* M$.
\end{defin}
Notice that the underlying object of $\ccat$ of the $\atrun{n-1}$-module $\Theta M$ can be described by the simple formula 

\begin{center}
$(\atrun{n-1} \oplus \Sigma F_{n}) \otimes _{\atrun{n-1}} M \simeq M \oplus (\atrun{0} \otimes_{\atrun{n-1}} M)[-n]$.
\end{center}
However, the above direct sum acquires an exotic $\atrun{n-1}$-module structure given by restricting scalars along the possibly non-zero derivation $d$.

The unit of the adjunction $p^{*} \dashv p_{*}$ induced by the projection $p: \atrun{n-1} \oplus \Sigma F_{n} \rightarrow \atrun{n-1}$ induces a map

\begin{center}
$d_{*} d_{0}^{*} M \rightarrow d_{*} p_{*} p^{*} d_{0} ^{*} M \simeq M$,
\end{center}
where in the last equivalence we have used that both $d$ and $d_{0}$ are sections of $p$. Thus, we obtain a natural transformation $\pi: \Theta M \rightarrow M$ of endofunctors of $\Mod_{\atrun{n-1}}$.

\begin{defin}
\label{defin:theta_sect}
We define the $\infty$-category $\ThetaSect_{\atrun{n-1}}$ to be the $\infty$-category of triples $(M, s, h)$, where $M \in \Mod_{\atrun{n-1}}$, $s:M \rightarrow \Theta M$, and $h$ is a homotopy from $\pi \circ s$ to $id_M$.
\end{defin}

Intuitively, $\ThetaSect_{\atrun{n-1}}$ is the $\infty$-category of $\atrun{n-1}$-modules $M$ equipped with a section $s$ of $\pi:\Theta M \to M$. However, it is important to remember that in the $\infty$-categorical setting, the homotopy realizing $\pi \circ s \simeq id_M$ is a necessary part of the data.

Note that the above definition given in terms of tuples is somewhat informal, because it only describes the objects rather than all of the $\infty$-categorical structure. For a more detailed description, see \cref{rem:construction_of_the_infty_categories_of_tuples}.

\begin{thm}
\label{thm:modules_over_higher_truncation_as_modules_with_a_section}
There is an equivalence of $\infty$-categories
\[	\Mod_{\atrun{n}} \simeq \ThetaSect_{\atrun{n-1}}.	\]
Moreover, under this equivalence, the extension of scalars $\Mod_{\atrun{n}} \rightarrow \Mod_{\atrun{n-1}}$ corresponds to the forgetful functor $\ThetaSect_{\atrun{n-1}} \to \atrun{n-1}$ sending $(M, s, h)$ to $M$. 
\end{thm}

\begin{proof}
By \cref{prop:pullback_square_of_module_infinity_categories}, we can identify $\Mod_{\atrun{n}}$ with the $\infty$-category of triples $(M, N, \alpha)$, where $M, N \in \Mod_{\atrun{n-1}}$ and $\alpha: d^{*} N \rightarrow d_{0}^* M$ is an equivalence. By adjunction, $\alpha$ corresponds to a map $\alpha^\prime: N \rightarrow \Theta M$ of $\atrun{n-1}$-modules. 

We claim that $\alpha$ is an equivalence if and only if the composite $\pi \circ \alpha^\prime: N \rightarrow \Theta M \rightarrow M$ is. To see this, notice that under the identification $p^{*} d^{*} N \simeq N$ and $p^{*} d_{0}^{*} M \simeq M$ coming from the fact that both $d, d_{0}$ are both sections of the projection $p: \atrun{n-1} \oplus \Sigma F_{n} \rightarrow \atrun{n-1}$, the composite $\pi \circ \alpha^\prime$ can be identified with $p^{*} \alpha:  p^{*} d^{*} M \rightarrow p^{*} d_{0}^* N$. Then, the claim follows immediately from the fact that $p^{*}$ is conservative, which is \cref{lemma:extension_along_0_equivalence_is_conservative}. It follows that $\Mod_{\atrun{n}}$ is equivalent to the $\infty$-category of triples 
\[  \dcat = \{(M, N, \alpha^\prime): M, N \in \Mod_{\atrun{n-1}}, \, \pi:N \to \Theta M\text{ such that }\pi \circ \alpha^\prime\text{ is an equivalence}\}.    \]

Consider the $\infty$-category of quintuples $\ecat = \{(M, N, \alpha^\prime, \beta, h)\}$, where $M, N, \alpha^\prime$ are as above, $\beta$ is a morphism $\beta: M \rightarrow N$ of $\atrun{n-1}$-modules and $h$ is a homotopy witnessing $\pi \circ \alpha^\prime \circ \beta \simeq id_{X}$. It is not hard to see that the functor 
\[  \ecat \to \dcat: (M, N, \alpha^\prime, \beta, h) \mapsto (M, N, \alpha^\prime)  \]
is fibred in spaces, that is, it is a Cartesian fibration all of whose fibres are $\infty$-groupoids. We claim that all of the fibres are contractible, so that this functor is an equivalence. Indeed, the fibre over any object $(M, N, \alpha^\prime)$ in the target is equivalent to the space of pairs 
\[	(\beta:M \to N, h:\pi \circ \alpha^\prime\circ  \beta \simeq id_M),	\]
which is precisely the homotopy fibre of the composition map
\[	(\pi\circ \alpha^\prime)_*:\Mod_{\atrun{n-1}}(M,N) \to \Mod_{\atrun{n-1}}(M,M)	\]
over the identity. Since $\pi\circ \alpha^\prime$ is an equivalence, this space is contractible. We deduce that the forgetful functor is an equivalence of $\infty$-categories, as promised. 

Consider the functor $F: \ecat \rightarrow \ThetaSect_{\atrun{n-1}}$ given by the formula 

\begin{center}
$F(M, N, \alpha^\prime, \beta, h) = (M, \alpha^\prime \circ \beta, h)$.
\end{center}
Similarly to the case above, one can show that the functor $F$ is fibred in spaces. Moreover, the fibre over any object $(M, s, h)$ of the target can be identified with the space of objects $N$ equipped with an equivalence $\beta: M \simeq N$, which is again contractible. We deduce that $F$ is also an equivalence, proving the theorem. 
\end{proof}

\begin{rem}
\label{rem:theta_lax_symmetric_monoidal}
If $A$ is commutative, the functor $\Theta: \Mod_{\atrun{n-1}} \rightarrow \Mod_{\atrun{n-1}}$ is lax symmetric monoidal, being a composite of two lax symmetric monoidal functors, and the $\infty$-category $ \ThetaSect_{\atrun{n-1}}$ of \cref{defin:theta_sect} carries a canonical symmetric monoidal structure. Under these conditions, the equivalence of \cref{thm:modules_over_higher_truncation_as_modules_with_a_section} can be promoted to an equivalence of symmetric monoidal $\infty$-categories. 
\end{rem}

\begin{rem}
\label{rem:construction_of_the_infty_categories_of_tuples}
Let us give the promised formal definition of the $\infty$-category $\ThetaSect_{\atrun{n-1}}$ of sections, whose objects are triples $(M, s, h)$, where $M \in \Mod_{\atrun{n-1}}$, $s: M \rightarrow \Theta M$ is a morphism and $h: \pi \circ s \simeq id_{M}$ is a homotopy. This construction of such $\infty$-categories is standard; the other $\infty$-categories of ``tuples'' appearing in the proof of \cref{thm:modules_over_higher_truncation_as_modules_with_a_section} are defined in the same way. 

First, define the $\infty$-category of tuples $\ThetaMorph_{\atrun{n-1}}$ as the pullback in the diagram

\begin{center}
	\begin{tikzpicture}
		\node (TL) at (0, 1.3) {$ \ThetaMorph_{\atrun{n-1}} $};
		\node (TR) at (4.9, 1.3) { $ \Fun(\partial \Delta^{1}, \Mod_{\atrun{n-1}}) $};
		\node (BL) at (0, 0) {$ \Mod_{\atrun{n-1}} $};
		\node (BR) at (4.9, 0) {$ \Fun(\Delta^{1}, \Mod_{\atrun{n-1}}) $};
		
		\draw [->] (TL) to (TR);
		\draw [->] (TL) to (BL);
		\draw [->] (TR) to (BR);
		\draw [->] (BL) to node[below] {$ (id, \Theta) $} (BR);
	\end{tikzpicture},
\end{center}
of simplicial sets, where the right vertical map is induced by the inclusion $\partial \Delta^{1} \hookrightarrow \Delta^{1}$. Objects of $\ThetaMorph_{\atrun{n-1}}$ are pairs $(M, s)$, where $M \in \Mod_{\atrun{n-1}}$ and $s: M \rightarrow \Theta M$ is a morphism.

Together, the upper horizontal arrow in the diagram above, the natural transformation $\pi$ and the identity determine a functor $\ThetaMorph_{\atrun{n-1}} \rightarrow \Fun(\partial \Delta^{2}, \Mod_{\atrun{n-1}})$ defined by 

\begin{center}
	\begin{tikzpicture}
		\node (A) at (-2, 0.7) {$ M $} ;

		\node (T) at (1, 1.4) {$ \Theta M $};
		\node (BL) at (0, 0) {$ M $};
		\node (BR) at (2, 0) {$ M $};
		
		\draw [->] (BL) to node[auto] {$ s $} (T);
		\draw [->] (T) to node[auto] {$ \pi $} (BR);
		\draw [->] (BL) to node[auto] {$ id $} (BR);
		\draw [|->] (-1.5, 0.7) to (-0.5, 0.7);
	\end{tikzpicture}.
\end{center}
We then define the $\infty$-category $\ThetaSect_{\atrun{n-1}}$ through the pullback 

\begin{center}
	\begin{tikzpicture}
		\node (TL) at (0, 1.3) {$ \ThetaSect_{\atrun{n-1}} $};
		\node (TR) at (4.9, 1.3) {$ \Fun(\Delta^{2}, \Mod_{\atrun{n-1}} )$};
		\node (BL) at (0, 0) {$ \ThetaMorph_{\atrun{n-1}} $};
		\node (BR) at (4.9, 0) {$ \Fun(\partial \Delta^{2}, \Mod_{\atrun{n-1}}) $};
		
		\draw [->] (TL) to (TR);
		\draw [->] (TL) to (BL);
		\draw [->] (TR) to (BR);
		\draw [->] (BL) to node[below] {$ (id, \Theta) $} (BR);
	\end{tikzpicture}.
\end{center}
The $0$-simplices of $\ThetaSect_{\atrun{n-1}}$ are the desired triples $(M, s, h)$.

Note that the way we constructed the relevant $\infty$-categories, both of  the forgetful functors $\ThetaSect_{\atrun{n-1}} \rightarrow \ThetaMorph_{\atrun{n-1}}$ and $\ThetaMorph_{\atrun{n-1}} \rightarrow \Mod_{\atrun{n-1}}$ are inner fibrations, since the inclusions $\partial \Delta^{1} \hookrightarrow \Delta^{1}$ and $\partial \Delta^{2} \hookrightarrow \Delta^{2}$ are cofibrations in the Joyal model structure. Thus, the coCartesian fibrations appearing in the proof of \cref{thm:modules_over_higher_truncation_as_modules_with_a_section} are in fact left fibrations in the sense of Lurie. 
\end{rem}

Lastly, we give a description of the $\infty$-category of modules over the $0$-truncation of $A$. Here, we assume that $A$ is a shift algebra in the sense of \cref{defin:shift_algebra} and that it admits ``enough'' periodic modules; this covers the cases we will need.  

\begin{thm}
\label{thm:modules_over_zero_truncation_same_as_the_derived_category}
Suppose that $A$ is a shift algebra in $\ccat$ such that $\Mod_{A}$ is generated under colimits by periodic modules. Then, the inclusion $\Mod_{\atrun{0}}^{\heartsuit} \hookrightarrow \Mod _{\atrun{0}}$ extends to an equivalence

\begin{center}
$\dcat(\Mod_{\atrun{0}}^{\heartsuit})_{\geq 0} \simeq \Mod_{\atrun{0}}$
\end{center}
between the connective derived $\infty$-category and modules over the $0$-truncation. Moreover, we have $\Mod_{\atrun{0}}^{\heartsuit} \simeq \Mod_{\pi_{0} A}(\ccat^{\heartsuit})$. 
\end{thm}

\begin{proof}
Since $\ccat$ is Grothendieck prestable, so is $\Mod_{\atrun{0}}$ and thus by \cite[C.5.4.11]{lurie_spectral_algebraic_geometry} it is enough to show that $\Mod_{\atrun{0}}$ is separated and $0$-complicial, that is, is generated under colimits by discrete objects. That it is separated immediately follows from the fact that the ambient $\infty$-category $\ccat$ itself is assumed to be separated. Moreover, since $\Mod_{A}$ is assumed to be generated under colimits by periodic modules $M$, the $\infty$-category $\Mod_{\atrun{0}}$ is generated under colimits by modules of the form $\atrun{0} \otimes _{A} M$ with $M$ periodic, which are all discrete by definition. 

The equivalence $\Mod_{\atrun{0}}^{\heartsuit} \rightarrow \Mod_{\pi_{0} A}(\ccat^{\heartsuit})$ is induced by the functor $\pi_{0}$. 
\end{proof}

\section{Linear obstruction theory}

In this section we introduce the notion of a potential $n$-stage for a periodic module and we develop the basic, linear variant of Goerss-Hopkins obstruction theory. 

Let us fix a graded symmetric monoidal, separated Grothendieck prestable $\infty$-category $\ccat$ and a shift algebra $A \in Alg(\ccat)$ in the sense of \cref{defin:shift_algebra}. Recall that an $A$-module $M$ is said to be \emph{periodic} if $\pi_{*} M \simeq \pi_{*} A \otimes _{\pi_{0} A} \pi_{0} M$; we will assume that $\Mod_{A}$ is generated under colimits by periodic modules. We will typically abbreviate the category of discrete $\pi_0A$-modules, $\Mod_{\pi_0A}(\ccat^\heartsuit)$, by $\Mod_{\pi_0A}$.

The linear Goerss-Hopkins theory yields obstructions to constructing a periodic $A$-module $M$ with prescribed $\pi_{0} M$ as a $\pi_{0} A$-module, with obstructions lying in $\Ext$-groups between discrete $\pi_{0} A$-modules. Intuitively, the latter arise as a consequence of \cref{thm:modules_over_zero_truncation_same_as_the_derived_category}, which states that the category of modules in $\ccat$ over the 0-truncation $\atrun{0}$ is equivalent to the derived category of discrete $\pi_{0}A$-modules, so that the mapping spaces therein compute $\Ext$-groups.

Recall that by \cref{prop:different_characterizations_of_periodic_objects} an $A$-module $M$ is periodic if and only if $\pi_{0} A \otimes _{A} M$ is discrete. This condition readily generalizes to modules over truncations of $A$, giving rise to the notion of a \emph{potential $n$-stage}. The following definition is classical, going back to the work of Blanc, Dwyer and Goerss on moduli of $\Pi$-algebras \cite{realization_space_of_a_pi_algebra}, \cite{pstrkagowski2017moduli}.

\begin{defin}
\label{defin:potential_n_stage}
We say an $\atrun{n}$-module $M$ is a \emph{potential $n$-stage for a periodic $A$-module} if $\atrun{0} \otimes _{\atrun{n}} M$ is discrete. We denote the $\infty$-category of potential $n$-stages by $\mathcal{M}_{n}$. 
\end{defin}
Notice that the above definition works for $0 \leq n \leq \infty$, where by abuse of notation we write $\atrun{\infty} := A$, so that a potential $\infty$-stage is the same as a periodic $A$-module and we have $\mathcal{M}_{\infty} \simeq \Mod_{A}^{per}(\ccat)$. On another extreme, a potential $0$-stage is the same as a discrete module over $A_{\leq 0}$, so that $\mathcal{M}_{0} \simeq \Mod_{\pi_{0}A}(\ccat^{\heartsuit})$. This suggests the correct intuition that potential stages for $0 < n < \infty$ interpolate between periodic $A$-modules and discrete $\pi_{0}A$-modules.

\begin{rem}
All objects of $\mathcal{M}_{n}$ are $n$-truncated, see \cref{lemma:homotopy_of_a_potential_n_stage_free} below, so that the $\infty$-category $\mathcal{M}_{n}$ of potential $n$-stages is an $(n+1)$-category.
\end{rem}

\begin{rem}
Since $\mathcal{M}_{0}$ is Grothendieck abelian, being equivalent to the category of modules in the Grothendieck abelian category $\ccat^{\heartsuit}$, it is tempting to guess that $\mathcal{M}_{n}$ for $n > 0$ is a Grothendieck abelian $(n+1)$-category in the sense of \cite[C.5.4.1]{lurie_spectral_algebraic_geometry}. This is not the case, as the $\infty$-categories $\mathcal{M}_{n}$ for $0 < n < \infty$ are usually not presentable, in fact need not admit finite limits. 
\end{rem}

Clearly, potential $n$-stages are taken to potential $m$-stages by the extension of scalars functor along $\atrun{n} \rightarrow \atrun{m}$ for $m \leq n$, so that there is an induced tower of $\infty$-categories 

\begin{center}
$\mathcal{M}_{\infty} \rightarrow \ldots \rightarrow \mathcal{M}_{n} \rightarrow \mathcal{M}_{n-1} \rightarrow \ldots \rightarrow \mathcal{M}_{0}$.
\end{center}

In practice, $\mathcal{M}_{\infty}$ can usually be identified with an $\infty$-category of geometric objects one wants to classify, $\mathcal{M}_{0}$ with the category of some algebraic objects, and the arrow $\mathcal{M}_{\infty} \rightarrow \mathcal{M}_{0}$ with a functor associated to some algebraic invariant. The moduli $\mathcal{M}_{n}$ for $0 < n < \infty$ stratify the loss of information present when we pass from a geometric object to an invariant of algebraic type. 

\begin{rem}
\label{rem:if_ambient_category_complete_infinite_stages_a_limit_of_finite_ones}
If $\ccat$ is complete, then the natural maps $\mathcal{M}_{\infty} \rightarrow \mathcal{M}_{n}$ induce an equivalence $\mathcal{M}_{\infty} \simeq \varprojlim \mathcal{M}_{n}$. To see this, notice that by \cref{prop:if_ambient_grothendieck_infty_cat_prestable_module_category_a_limit_of_modules_over_truncation} we have $\Mod_{A} \simeq \varprojlim \Mod_{\atrun{n}}$, and one sees immediately that this equivalence restricts to the one we need. 
\end{rem}

We will now use the relation between the $\infty$-categories of modules over the truncations of the unit we derived in the previous section to develop an obstruction theory to extending a potential $(n-1)$-stage to a potential $n$-stage.

\begin{lemma}
\label{lemma:homotopy_of_a_potential_n_stage_free}
Let $M \in \Mod_{\atrun{n}}$ and $n < \infty$. Then, $M$ is a potential $n$-stage if and only if $\pi_{*} M \simeq \pi_{*} \atrun{n} \otimes _{\pi_{0} A} \pi_{0} M$. 
\end{lemma}

\begin{proof}
This proof is analogous to the one given in \cref{prop:different_characterizations_of_periodic_objects}. In more detail, we have a cofibre sequence
\begin{equation*}
\Sigma^{n} \pi_{0} A [-n] \rightarrow \atrun{n} \rightarrow \atrun{n-1}
\end{equation*}
of right $\atrun{n}$-modules, in which the first map is induced from $\tau^{n}: \Sigma^{n} A[-n] \rightarrow A$ by applying $n$-truncation. Tensoring this with $M$ gives a cofibre sequence
\begin{equation*}
\Sigma^{n}(\pi_0A \otimes_{\atrun{n}} M)[-n] \to M \to \atrun{n-1} \otimes_{\atrun{n}} M.
\end{equation*}
Assume that $M$ is a potential $n$-stage, and assume for induction that the statement has been proved for $k < n$, the case of $n=0$ being trivial. Then in particular, we have
\[	\pi_*(\atrun{n-1} \otimes_{\atrun{n}} M) \simeq \pi_*\atrun{n-1} \otimes _{\pi_{0}A} \pi_0M.	\]
This is concentrated in homotopy degrees 0 through $n-1$, and likewise, using discreteness of $\pi_0A \otimes_{\atrun{n}} M$,
\[	\pi_*\Sigma^{n}(\pi_0A \otimes_{\atrun{n}} M)[-n] = \pi_0M[-n],	\]
concentrated in homotopy degree $n$. This implies the needed statement about homotopy groups.

Conversely, the statement about homotopy groups implies that the fibre of
\[	M \to \atrun{n-1} \otimes_{\atrun{n}} M	\]
is concentrated in homotopy degree $n$, which means that $\pi_0A \otimes_{\atrun{n}} M$ is discrete.
\end{proof}

\begin{rem}
\label{rem:any_potential_n_stage_is_an_n_type_and_defin_equivalent_to_classical_one}
It follows from \cref{lemma:homotopy_of_a_potential_n_stage_free} that our definition of a potential $n$-stage is equivalent to the classical one, as appearing in the work of Goerss and Hopkins \cite[3.3.1]{moduli_problems_for_structured_ring_spectra}. Note that there is no distinction between $\atrun{n}$ and $A$-modules in the present case, since by the above any potential $n$-stage is $n$-truncated and it follows by \cref{lemma:any_k_truncated_module_canonically_a_module_over_k_truncation} that its $\atrun{n}$-module structure is uniquely determined by the underlying $A$-module. 
\end{rem}

Recall from \cref{thm:modules_over_higher_truncation_as_modules_with_a_section} that we can identify $\Mod_{\atrun{n}}$ with the $\infty$-category $\ThetaSect_{\atrun{n-1}}$ of triples $(M, s, h)$,  where $M \in \Mod_{\atrun{n-1}}$ and $s: M \rightarrow \Theta M$ and $h$ is a distinguished homotopy $id_{M} \simeq \pi \circ s$ witnessing that $s$ is a section of the natural projection $\pi: \Theta M \rightarrow M$. 

Here, the source $\Theta M$, as defined in \cref{defin:hopkins_lurie_theta_functor}, is given as an object of $\ccat$ by 

\begin{center}
$\Theta M \simeq M \oplus (\Sigma^{n+1} \atrun{0}[-n] \otimes _{\atrun{n-1}} M)$,
\end{center}
but it comes with a possibly exotic $\atrun{n-1}$-module structure induced by restricting scalars along a possibly non-trivial derivation $d: \atrun{n-1} \rightarrow \atrun{n-1} \oplus \Sigma^{n+1} \atrun{0}[-n]$. 

We start by making a simple observation about the natural transformation $\pi$ from $\Theta$ to identity introduced in the discussion following \cref{defin:hopkins_lurie_theta_functor}.

\begin{prop}
\label{prop:projection_from_theta_n_connective}
Let $M \in \Mod_{\atrun{n-1}}$. Then, $\pi: \Theta M \rightarrow M$ is an $n$-equivalence.
\end{prop}

\begin{proof}
Recall that the map $\pi$ is induced by the unit of the adjunction induced by the projection

\begin{center}
$p:  (\atrun{n-1} \oplus \Sigma^{n+1} \atrun{0} [-n]) \rightarrow \atrun{n-1}$. 
\end{center}
Then, it follows that $\pi$ is an $n$-equivalence since $p$ is, and $n$-equivalences are preserved by the relative tensor product because it can be computed using the bar construction.
\end{proof}

\begin{cor}
\label{cor:cofibre_of_map_from_theta_for_potential_stages}
Let $M \in \mathcal{M}_{n-1}$ be a potential $(n-1)$-stage. Then the cofibre of $\pi: \Theta M \rightarrow M$ is equivalent to $\Sigma^{n+2} \pi_{0}M [-n]$.
\end{cor}

\begin{proof}
We've seen that as an object of $\ccat$, $\Theta M \simeq M \oplus \Sigma^{n+1} (\atrun{0} \otimes _{\atrun{n-1}} M)$, and since $M$ is assumed to be a potential $(n-1)$-stage, the latter summand is a suspension of a discrete one. Since the map $\Theta M \rightarrow M$ is an $n$-equivalence by \cref{prop:projection_from_theta_n_connective}, the cofibre must be of the form given above. 
\end{proof}

\begin{thm}
\label{thm:obstructions_in_ext_groups_to_lifting_a_potential_stage}
Let $M \in \mathcal{M}_{n-1}$ be a potential $(n-1)$-stage. Then, there exists an obstruction $o_{M} \in \Ext_{\pi_{0}A} ^{n+2, n}(\pi_{0} M, \pi_{0} M)$ which vanishes if and only if $M$ can be lifted to a potential $n$-stage, that is, if there exists $M' \in \mathcal{M}_{n}$ with $M \simeq \atrun{n-1} \otimes _{\atrun{n}} M'$.
\end{thm}

\begin{proof}
By \cref{thm:modules_over_higher_truncation_as_modules_with_a_section}, $M$ extends to a potential $n$-stage if and only if there exists a section of $\Theta M \rightarrow M$. Then, by \cref{cor:cofibre_of_map_from_theta_for_potential_stages}, there is a cofibre sequence 

\begin{center}
$\Theta M \rightarrow M \rightarrow \Sigma^{n+2} \pi_{0} M [-n]$
\end{center}
and the first map admits a section if and only if the second is zero. Thus, the obstruction to the existence of a section lies in
\begin{center}
$\map_{\atrun{n-1}}(M, \Sigma^{n+2} \pi_{0}M [-n]) \simeq \map_{\pi_{0}A}(\pi_{0}A \otimes _{\atrun{n-1}} M, \Sigma^{n+2} \pi_{0}M [-n])$, 
\end{center}
where we have used that the target is canonically a $\pi_{0}A$-module, because it is a suspension of a discrete object. 

Since $M$ is a potential $(n-1)$-stage, we have an equivalence $\pi_{0}A \otimes _{\atrun{n-1}} M \simeq \pi_{0}M$ and so we can further rewrite the above mapping space as 
\[	\map_{\pi_{0}A}(\pi_{0}M, \Sigma^{n+2} \pi_{0}M [-n]).	\]
Through the equivalence $\Mod_{\atrun{0}}(\ccat) \simeq \dcat(\Mod_{\pi_{0}A}(\ccat^{\heartsuit}))_{\geq 0}$ of \cref{thm:modules_over_zero_truncation_same_as_the_derived_category}, this can be identified with the $\Ext$-group as in the statement of the theorem. 
\end{proof}

\begin{cor}
\label{cor:if_c_is_complete_then_we_have_an_obstruction_theory_to_constructing_a_periodic_a_module}
Suppose that $\ccat$ is complete. Then, for any $M \in \Mod_{\pi_{0}A}(\ccat^{\heartsuit})$ there exists a sequence of inductively defined obstructions in $\Ext_{\pi_0A} ^{n+2, n}(M, M)$, where $n \geq 1$, which vanish if and only if there exists a periodic $A$-module $P$ with $\pi_{0} P \simeq M$ as a $\pi_{0}A$-module.
\end{cor}

\begin{proof}
Since $\ccat$ is complete, we have $\mathcal{M}_{\infty} \simeq \varprojlim \mathcal{M}_{n}$ by \cref{rem:if_ambient_category_complete_infinite_stages_a_limit_of_finite_ones}.  It then follows that the $\infty$-category of periodic $A$-modules satisfying the above condition can be identified with 

\begin{center}
$\mathcal{M}_{\infty} \times _{\mathcal{M}_{0}} \{ M \} \simeq \varprojlim \mathcal{M}_{n} \times _{\mathcal{M}_{0}} \times \{ M \}$. 
\end{center}
It follows that to a construct a point $P$ in this space it is enough to give a compatible sequence $M_{n} \simeq \mathcal{M}_{n}$ such that $M_{0} \simeq M$. The statement is then immediate from \cref{thm:obstructions_in_ext_groups_to_lifting_a_potential_stage}.
\end{proof}

\begin{rem}
One has to be careful, as the obstruction $o_{M} \in \Ext_{\pi_0A}^{n+2, n}(\pi_{0} M, \pi_{0} M)$ of \cref{thm:obstructions_in_ext_groups_to_lifting_a_potential_stage} associated to a potential $(n-1)$-stage $M \in \mathcal{M}_{n-1}$ is not canonical.

To see this, notice that the obstruction depends on the choice of an equivalence between the cofibre of $\Theta M \rightarrow M$ and $\Sigma^{n+2} \pi_{0}M [-n]$, so that the obstruction is only well-defined up to the action of the automorphism group $Aut_{\Mod_{\pi_{0}A}}(\pi_{0} M)$. Note that, in particular, whether the obstruction vanishes does not depend on that choice. 
\end{rem}

Now that we've given an obstruction to lifting a potential $(n-1)$-stage to a potential $n$-stage, we will also relate the mapping spaces between the two.

\begin{prop}
\label{prop:fibre_sequence_of_mapping_spaces_of_linear_potential_stages}
Let $M, N \in \mathcal{M}_{n}$ be potential $n$-stages and let $uM := \atrun{n-1} \otimes _{\atrun{n}} M$ and $uN := \atrun{n-1} \otimes _{\atrun{n}} N$ denote their images in $\mathcal{M}_{n-1}$. Then, there's a fibre sequence

\begin{center}
$\map_{\mathcal{M}_{n}}(M, N) \rightarrow \map_{\mathcal{M}_{n-1}}(uM, uN) \rightarrow \map_{D(\Mod_{\pi_{0}A}(\ccat^{\heartsuit}))}(\pi_{0} M, \Sigma^{n+1} \pi_{0} N [-n])$,
\end{center}
where on the right we have the mapping space in the derived category.
\end{prop}

\begin{proof}
By \cref{rem:any_potential_n_stage_is_an_n_type_and_defin_equivalent_to_classical_one}, the middle and left mapping spaces can be computed in $A$-modules. Moreover, \cref{lemma:homotopy_of_a_potential_n_stage_free} implies that the maps $M \rightarrow uM$ and $N \rightarrow uN$ exhibit their targets as the $(n-1)$-truncation of the source, so that we can replace $uM, uN$ by $M_{\leq n-1}, N_{\leq n-1}$. 

By adjunction, we have $\map_{A}(M, N_{\leq n-1}) \simeq \map_{A}(M_{\leq n-1}, N_{\leq n-1})$, and the left arrow in the statement can be identified with the morphism $\map_{A}(M, N) \rightarrow \map_{A} (M, N_{\leq n-1})$ induced by the truncation $N \rightarrow N_{\leq n-1}$. It follows that there's a fibre sequence 

\begin{center}
$\map_{A}(M, N) \rightarrow \map_{A}(M, N_{\leq n-1}) \rightarrow \map_{A}(M, \Sigma F)$,
\end{center}
where $F$ is the fibre of $N \rightarrow N_{\leq n-1}$. To prove the statement, it is thus enough to identify the right mapping space with the one computed in the derived $\infty$-category. 

Again, by \cref{lemma:homotopy_of_a_potential_n_stage_free} we see that $F \simeq \Sigma^{n} \pi_{0} N[-n]$, which is a suspension of a discrete object and so canonically a $\pi_{0}A$-module. We deduce that there's an equivalence

\begin{center}
$\map_{A}(M, \Sigma F) \simeq \map_{\pi_{0}A}(\pi_{0}A \otimes _{\atrun{n}} M, \Sigma^{n+1} \pi_{0}N [-n]) \simeq \map_{\pi_{0}A}(\pi_{0}M, \Sigma^{n+1} \pi_{0}N [-n])$
\end{center}
where the first equivalence is an application of adjunction, and in the second one we use that $M$ is a potential $n$-stage. Then, one sees that under the equivalence of \cref{thm:modules_over_zero_truncation_same_as_the_derived_category} the above mapping space corresponds to the one in the statement of the proposition.
\end{proof}

\begin{cor}\label{cor:mapping_space_spectral_sequence_linear}
Suppose that $\ccat$ is complete. Let $M$, $N$ be periodic $A$-modules. Then there is a spectral sequence of signature
\[	E_1^{s,t} = \Ext^{2s-t,s}_{\pi_0A}(\pi_0M, \pi_0N), \quad t \ge s \ge 0,	\]
converging conditionally to $\pi_*\map_A(M,N)$.
\end{cor}

\begin{proof}
Let 
\[	\mathrm{maps}_s := \map_{\atrun{s}}(M \otimes_{A} \atrun{s}, N \otimes_{A} \atrun{s}).	\]
By the previous proposition, the fiber of $\mathrm{maps}_s \to \mathrm{maps}_{s-1}$ is
\[	F_s \simeq \map_{D(\pi_0A)}(\pi_0M, \Sigma^s \pi_0N[-s]).	\]
Since these are mapping spaces in the derived category of an abelian category, they are all topological abelian groups abelian groups, and in particular all pointed by the zero map. The Bousfield-Kan spectral sequence for the homotopy groups of a tower \cite{bousfield1972homotopy} then takes the form
\[	E_1^{s,t} = \pi_{t-s}F_s = \Ext^{2s-t,s}_{\pi_0A}(\pi_0M, \pi_0N), \quad t \ge s \ge 0.	\]
Since $\ccat$ is complete,
\[	\map_A(M,N) = \lim_s \mathrm{maps}_s,	\]
so the spectral sequence converges conditionally to $\pi_*\map_A(M,N)$.
\end{proof}

\begin{rem}\label{rem:conditional_convergence}
For conditional convergence, see \cite{boardman1999conditionally} or \cite[IX.5.4]{bousfield1972homotopy}. The spectral sequence converges completely to $\pi_*\map_A(M,N)$ if the derived limit
\[	\lim{}^1 (\dotsb \subseteq E_r^{r+2,r+2+i} \subseteq E_r^{r+1,r+1+i})	\]
vanishes for each fixed $r$ and $i$. This condition holds in a wide variety of cases, for example, if each $E_r^{s,t}$ is finitely generated for some fixed $r$.
\end{rem}

\begin{rem}
Our two results about the functor $\mathcal{M}_{n} \rightarrow \mathcal{M}_{n-1}$, namely the obstructions to lifting of \cref{thm:obstructions_in_ext_groups_to_lifting_a_potential_stage} and the relation between the mapping spaces of \cref{prop:fibre_sequence_of_mapping_spaces_of_linear_potential_stages}, are combined in the work of Goerss and Hopkins into a single theorem by constructing a pullback square involving the $\infty$-groupoids of potential $n$-stages and certain spaces of algebraic nature \cite[3.3.5]{moduli_problems_for_structured_ring_spectra}. We decided to separate these two results for clarity of exposition.
\end{rem}

\begin{rem}
We decided not to write at this level of generality, but if one is only interested in linear obstruction theory, then the shift algebra $A$ can be dispensed with entirely. What suffices is to have a complete, graded prestable $\infty$-category $\ccat$ together with a \emph{thread structure}; that is, a transformation $\tau: \Sigma M[-1] \rightarrow M$ natural in $M \in \ccat$. This is enough to characterize periodic objects and one should again ask that $\ccat$ is generated by such under colimits.

One can show as in \cref{thm:modules_over_zero_truncation_same_as_the_derived_category} that the $\infty$-categories $\Mod_{\atrun{n}}(\ccat)$ considered above can be identified with the universal $n$-complicial approximation to the prestable $\infty$-category $\ccat$ of Lurie, see \cite[C.5]{lurie_spectral_algebraic_geometry}.  The latter notion makes sense without the presence of the algebra $A$, providing the needed tower of $\infty$-categories. Our appeal to deformation theory of algebras needs to be replaced by an explicit calculation as in \cite[3.3]{pstrkagowski2017moduli}, but the results contained in this section, in particular, \cref{cor:if_c_is_complete_then_we_have_an_obstruction_theory_to_constructing_a_periodic_a_module} and \cref{cor:mapping_space_spectral_sequence_linear}, hold without any change. 
\end{rem}

\section{Goerss-Hopkins obstruction theory}

In this section we give a multiplicative extension of the linear theory developed in the previous one, constructing obstructions to realizing an algebra in $\Mod_{\pi_{0}A}(\ccat^{\heartsuit})$ by a periodic $\mathbf{E}_{k}$-algebra in $\Mod_{A}$. The obstructions will lie in suitable Andr\'{e}-Quillen cohomology groups of $\pi_{0} A$-algebras. 

Again, we assume that $\ccat$ is a graded symmetric monoidal, separated Grothendieck prestable $\infty$-category. We fix a \emph{commutative} shift algebra $A \in \textnormal{CAlg}(\ccat)$ such that $\Mod_{A}$ is generated under colimits by periodic modules.

Our arguments will follow rather closely the linear version established before, with the important difference that the arguments involving the prestability of the $\infty$-categories of modules will need to be replaced by the deformation theory of algebras.

\begin{defin}
\label{defin:a_potential_n_stage_for_an_ek_algebra}
A \emph{potential $n$-stage for an $\ekoperad{k}$-algebra} is an $\ekoperad{k}$-algebra $R$ in $\Mod_{\atrun{n}}$ whose underlying $\atrun{n}$-module is a potential $n$-stage. We denote the $\infty$-category of potential $n$-stages for an $\ekoperad{k}$-algebra by $\textnormal{Alg}_{\ekoperad{k}}(\mathcal{M}_{n})$.
\end{defin}
Being more verbose, one could say that $R$ is a \emph{potential $n$-stage to realizing the $\pi_{0}A$-algebra $\pi_{0} R$ as homotopy of a periodic $\mathbf{E}_{k}$-algebra}, but we will refrain from doing so. Note that the notation $\textnormal{Alg}_{\ekoperad{k}}(\mathcal{M}_{n})$ is slightly abusive, since the $\infty$-category $\mathcal{M}_{n}$ of potential $n$-stages is not symmetric monoidal: by $\textnormal{Alg}_{\ekoperad{k}}(\mathcal{M}_{n})$, we really mean $\textnormal{Alg}_{\ekoperad{k}}(\Mod_{\atrun{n}}) \times _{\Mod_{\atrun{n}}} \mathcal{M}_{n}$. 

Going back to \cref{defin:potential_n_stage}, we see that a potential $n$-stage $R$ for an $\ekoperad{k}$-algebra is an object $R \in \textnormal{Alg}_{\ekoperad{k}}(\Mod_{\atrun{n}})$ such that $\pi_{0}A \otimes _{\atrun{n}} R$ is discrete. In this case $\pi_{0} A \otimes _{\atrun{n}} R \simeq \pi_{0} R$ is an algebra in $\Mod_{\pi_{0} A}(\ccat^{\heartsuit})$, commutative if $k \geq 2$. 

\begin{rem}
\label{rem:multiplicative_potential_infty_stages_limit_of_finite_ones}
If $\ccat$ is complete, so that $\mathcal{M}_{\infty} \simeq \varprojlim \mathcal{M}_{n}$ by \cref{rem:if_ambient_category_complete_infinite_stages_a_limit_of_finite_ones}, then we also have $\textnormal{Alg}_{\ekoperad{k}}(\mathcal{M}_{\infty}) \simeq \textnormal{Alg}_{\ekoperad{k}}(\mathcal{M}_{n})$ by the same argument.
\end{rem}

Recall that by \cref{thm:modules_over_higher_truncation_as_modules_with_a_section}, there's an equivalence between $\Mod_{\atrun{n}}$ and the $\infty$-category $\ThetaSect_{A_{n-1}}$ of objects $M \in \Mod_{\atrun{n-1}}$ equipped with a section $s: M \rightarrow \Theta M$.of the natural projection $\pi:\Theta M \to M$. 

The above equivalence can be promoted to a symmetric monoidal one, see \cref{rem:theta_lax_symmetric_monoidal}, and so extends to an equivalence between $\ekoperad{k}$-algebras in $\Mod_{\atrun{n}}$ and $\ekoperad{k}$-algebras in $\Mod_{\atrun{n-1}}$ equipped with a section which is also a map of $\ekoperad{k}$-algebras. Thus, we can determine which potential $(n-1)$-stages for $\ekoperad{k}$-algebras lift to potential $n$-stages by giving sufficient and necessary conditions for such a section to exist.

\begin{lemma}
\label{lemma:projection_from_theta_is_a_square_zero_extension}
Let $R$ be a potential $(n-1)$-stage for an $\ekoperad{k}$-algebra. Then, the map $\pi:\Theta R \rightarrow R$ is a square-zero extension of the $\ekoperad{k}$-$\atrun{n-1}$-algebra $R$ by the $R$-$\ekoperad{k}$-module $\Sigma^{n+1} \pi_{0}R [-n]$.
\end{lemma}

\begin{proof}
It is clear that $\pi$ is an $\ekoperad{k}$-$\atrun{n-1}$-algebra map, and it follows from \cref{cor:cofibre_of_map_from_theta_for_potential_stages} that its fibre is $\Sigma^{n+1} \pi_{0}R [-n]$. As this fibre has homotopy concentrated in a single positive degree, $\pi$ is a square-zero extension, by \cite[7.4.1.26]{higher_algebra}.
\end{proof}

The $\ekoperad{k}$-$\atrun{n-1}$-algebra $R$ has a cotangent complex $\LL_{R/\atrun{n-1}}^{\ekoperad{k}}$, see \cite[7.3, 7.4]{higher_algebra}, which has the property that there exists a natural equivalence of $\infty$-groupoids
\begin{equation*}
\left\{ \parbox{.5\textwidth}{Square-zero extensions of $\ekoperad{k}$-$\atrun{n-1}$-algebras $\widetilde{R} \to R$ by the $R$-module $M$} \right\}
\quad \simeq \quad \mathrm{map}^{\ekoperad{k}}_R(\LL_{R/\atrun{n-1}}^{\ekoperad{k}}, \Sigma M),
\end{equation*}
where the mapping space on the right is computed in $\ekoperad{k}$-$R$-modules. In particular, an extension $\widetilde{R} \to R$ is split, that is, admits a section, if and only if the classifying map from the cotangent complex to the suspension of the fibre is null. 

\begin{thm}[Goerss-Hopkins, obstructions to lifting objects]
\label{thm:obstructions_to_lifting_in_multiplicative_case}
Let $R$ be a potential $(n-1)$-stage for an $\ekoperad{k}$-algebra. Then, there exists an obstruction in the Andr\'e-Quillen cohomology group
\begin{center}
$\Ext^{n+2, n}(\mathbb{L}_{\pi_{0} R}^{\ekoperad{k}}, \pi_{0} R)$
\end{center}
which vanishes if and only if $R$ can be lifted to a potential $n$-stage, where the $\Ext$-group is computed in the $\infty$-category of $\ekoperad{k}$-modules over $\pi_0R$ in the derived $\infty$-category $\dcat(\Mod_{\pi_0A}(\ccat^\heartsuit))$.
\end{thm}

\begin{proof}
By the equivalence of \cref{thm:modules_over_higher_truncation_as_modules_with_a_section}, we know $R$ can be lifted to a potential $n$-stage if and only if $\pi: \Theta R \rightarrow R$ admits a section. By \cref{lemma:projection_from_theta_is_a_square_zero_extension}, $\pi$ is a square-zero extension of $\atrun{n-1}$-algebras and so is classified by an element of 

\begin{center}
$\pi_{0} \map^{\ekoperad{k}}_{R} (\mathbb{L}_{R / \atrun{n-1}}^{\ekoperad{k}}, \Sigma^{n+2} \pi_{0}R [-n])$ 
\end{center}
which vanishes if and only if there exists a section.

Notice that the target of this mapping space is a suspension of a discrete object and so is canonically a $\pi_{0} A$-module. Thus, we have
\begin{align*}
	\pi_{0} \map^{\ekoperad{k}}_{R} (\mathbb{L}_{R / \atrun{n-1}}^{\ekoperad{k}}, \Sigma^{n+2} \pi_{0} R [-n]) &\cong \pi_0 \map^{\ekoperad{k}}_{\pi_0 R}(\pi_0 A \otimes_{\atrun{n-1}} \LL_{R/\atrun{n-1}}^{\ekoperad{k}}, \Sigma^{n+2} \pi_0 R[-n]),
\end{align*}
where we have used that $R$ is a potential $(n-1)$-stage so that $\pi_{0} A \otimes_{\atrun{n-1}} R \simeq \pi_{0}R$, and further 

\begin{center}
$\pi_0 \map^{\ekoperad{k}}_{\pi_{0}R}(\pi_0 A \otimes_{\atrun{n-1}} \LL_{R/\atrun{n-1}}^{\ekoperad{k}}, \Sigma^{n+2} \pi_0 R[-n]) \simeq \map^{\ekoperad{k}}_{\pi_0R}(\LL_{\pi_0 R/\pi_0 A}^{\ekoperad{k}}, \Sigma^{n+2}\pi_0R[-n])$
\end{center}
by the base change formula for the cotangent complex of \cite[7.3.3.7]{higher_algebra}. Finally, the path components of the above mapping space correspond to the $\Ext$-group given in the statement of the theorem under the equivalence $\dcat(\Mod_{\pi_{0} A}(\ccat^{\heartsuit})) \simeq \Mod_{\pi_{0} A}(\ccat)$ of  \cref{thm:modules_over_zero_truncation_same_as_the_derived_category}.
\end{proof}

\begin{cor}
\label{cor:if_c_is_complete_we_have_inductive_obstructions_to_existence_of_periodic_algebras}
Suppose that $\ccat$ is complete and let $S$ be a $\pi_{0} A$-algebra in $\ccat^{\heartsuit}$, commutative if $k \geq 2$. Then, there exists a sequence of inductively defined obstructions in $\Ext^{n+2, n}(\mathbb{L}_{S}^{\ekoperad{k}}, S)$, where $n \geq 1$ and the extensions are computed in $\ekoperad{k}$-$\pi_{0}R$-modules in $\dcat(\Mod_{\pi_{0} A}(\ccat^{\heartsuit}))$, which vanish if and only if there exists a periodic $\ekoperad{k}$-$A$-algebra $R$ such that $\pi_{0} R \simeq S$ as $\pi_{0} A$-algebras. 
\end{cor}

\begin{proof}
As in the linear case covered in \cref{cor:if_c_is_complete_then_we_have_an_obstruction_theory_to_constructing_a_periodic_a_module}, this is immediate from \cref{thm:obstructions_to_lifting_in_multiplicative_case} and $\textnormal{Alg}_{\ekoperad{k}}(\mathcal{M}_{\infty}) \times _{\textnormal{Alg}_{\ekoperad{k}}(\mathcal{M}_{0})} \{ S \} \simeq \varprojlim \textnormal{Alg}_{\ekoperad{k}}(\mathcal{M}_{n}) \times _{\textnormal{Alg}_{\ekoperad{k}}(\mathcal{M}_{0})} \{ S \}$, which is a consequence of \cref{rem:multiplicative_potential_infty_stages_limit_of_finite_ones}.
\end{proof}

\begin{rem}
\label{rem:obstructions_in_derived_category_in_commutative_case}
In the particular case of the commutative operad, an $\ekoperad{\infty}$-$\pi_{0}R$-module is the same as a left $\pi_{0} R$-module, so that the obstructions of \cref{thm:obstructions_to_lifting_in_multiplicative_case} live in $\Ext$-groups computed in the $\infty$-category of left $\pi_{0} R$-modules in $\dcat(\Mod_{\pi_{0}A}(\ccat^{\heartsuit}))$. Under very weak technical conditions, the latter admits a much easier description.

That is, suppose that either $\pi_{0} R$ is flat over $\pi_{0}A$ or that $\Mod_{A}(\ccat)$ is generated by periodic $A$-modules $M$ such that $\pi_{0} M$ is flat over $\pi_{0} A$. In either case, it follows that 
\begin{center}
$\pi_{0} R \otimes M \simeq (\pi_{0} R\otimes _{\pi_{0}A} \pi_{0}A) \otimes M \simeq \pi_{0} R \otimes _{\pi_{0}A} M$
\end{center}
is discrete, so that the $\infty$-category $\Mod_{\pi_{0}R}(\ccat)$ is generated by discrete objects. Then, the same argument as in the proof of \cref{thm:modules_over_zero_truncation_same_as_the_derived_category} shows that $\Mod_{\pi_{0}R}(\ccat) \simeq \dcat(\Mod_{\pi_{0}R}(\ccat^{\heartsuit}))_{\geq 0}$. 

Using the aforementioned theorem, one rewrites 

\begin{center}
$\Mod_{\pi_{0}R}(\dcat(\Mod_{\pi_{0}A}(\ccat^{\heartsuit}))) \simeq \Mod_{\pi_{0}R}(\Mod_{\pi_{0}A}(\ccat)) \simeq \Mod_{\pi_{0}R}(\ccat)$, 
\end{center}
which combined with the equivalence of the previous paragraph shows that the obstructions of \cref{thm:obstructions_to_lifting_in_multiplicative_case} can be in fact computed in $\dcat(\Mod_{\pi_{0}R}(\ccat^{\heartsuit}))$, the derived $\infty$-category of discrete $\pi_{0}R$-modules. 
\end{rem}

\begin{prop}[Goerss-Hopkins, obstruction to lifting maps]
\label{prop:relation_between_mapping_spaces_between_potential_stages_in_the_multiplicative_case}
Let $R, S$ be potential $n$-stages for an $\ekoperad{k}$-algebra and $\phi: uR \rightarrow uS$ be a map of corresponding potential $(n-1)$-stages. Then, there's an equivalence

\begin{center}
$F_{\phi} \simeq P_{0, \phi^{\prime}} \map_{\pi_{0}R} ^{\ekoperad{k}}(\mathbb{L}_{\pi_{0}R / \pi_{0} A}^{\ekoperad{k}}, \Sigma^{n+1} \pi_{0} S[-n])$,
\end{center}
between the fibre over $\phi$ of $\map_{\textnormal{Alg}_{\ekoperad{k}}(\mathcal{M}_{n})}(R, S) \rightarrow \map_{\textnormal{Alg}_{\ekoperad{k}}(\mathcal{M}_{n-1})}(uR, uS)$ and the space of paths between $0$ and a certain morphism $\phi^{\prime}: \mathbb{L}_{\pi_{0}R / \pi_{0} A}^{\ekoperad{k}} \rightarrow \Sigma^{n+1} \pi_{0} S[-n]$.
\end{prop}

\begin{proof}
By \cref{rem:any_potential_n_stage_is_an_n_type_and_defin_equivalent_to_classical_one}, the mapping spaces in $\textnormal{Alg}_{\ekoperad{k}}(\mathcal{M}_{n})$ and $\textnormal{Alg}_{\ekoperad{k}}(\mathcal{M}_{n-1})$ can be computed in $\ekoperad{k}$-$A$-algebras. Then, \cref{lemma:homotopy_of_a_potential_n_stage_free} implies that the natural maps $R \rightarrow uR$ and $S \rightarrow uS$ identify their targets with Postnikov truncation of the former, so that we can replace $uR, uS$ by $R_{\leq n-1}$ and $S_{\leq n-1}$. Thus, the fibre we're trying to describe can be identified with the fibre of $\map_{\textnormal{Alg}_{\ekoperad{k}}(\Mod_{A})}(R, S) \rightarrow \map_{\textnormal{Alg}_{\ekoperad{k}}(\Mod_{A})}(R, S_{\leq n-1})$.

Again, by \cite[7.4.1.26]{higher_algebra} $S$ is a square-zero extension of $S_{\leq n-1}$ by $\Sigma^{n} \pi_{0}S [-n]$. Then, by \cite[7.4.1.8]{higher_algebra} the relevant fibre can be described as the space of paths 

\begin{center}
$P_{0, \widetilde{\phi}} \map_{R}^{\ekoperad{k}} (\mathbb{L}^{\ekoperad{k}}_{R / \atrun{n}}, \Sigma^{n+1} S_{\leq 0}[-n])$,
\end{center}
where $\widetilde{\phi}$ is the element classifying the square-zero extension $S \times _{S_{\leq n-1}} R \rightarrow R$. 

Observe that the target of the above mapping space is canonically a $\pi_{0} A$-module, so that using the base change formula for the cotangent complex and $\pi_{0} A \times _{\atrun{n}} R \simeq \pi_{0} R$ we can rewrite this mapping space as $\map_{\pi_{0}R} ^{\ekoperad{k}} (\mathbb{L}^{\ekoperad{k}}_{\pi_{0} R / \pi_{0} A}, \Sigma^{n+1} \pi_{0} S[-n])$, which ends the argument.
\end{proof}

The statement of \cref{prop:relation_between_mapping_spaces_between_potential_stages_in_the_multiplicative_case}, involving the path space of a mapping space, can be perhaps a little mysterious at first sight. However, it immediately yields obstructions to lifting maps as well as control over the space of possible lifts, see \cref{rem:obstruction_to_lifting_morphism_of_potential_stages} and \cref{rem:fibre_of_map_between_mapping_spaces_of_potential_stages} below. 

\begin{rem}
\label{rem:obstruction_to_lifting_morphism_of_potential_stages}
Notice that \cref{prop:relation_between_mapping_spaces_between_potential_stages_in_the_multiplicative_case} implies that if $R, S$ are potential $n$-stages for $\ekoperad{k}$-algebras and $\phi: uR \rightarrow uS$ is a map of the corresponding potential $(n-1)$-stages, then $\phi$ lifts to a map $R \rightarrow S$ if and only if the corresponding element $\phi^{\prime}: \mathbb{L}_{\pi_{0}R / \pi_{0} A}^{\ekoperad{k}} \rightarrow \Sigma^{n+1} \pi_{0} S[-n]$ is nullhomotopic. Thus, the homotopy class of the latter determines an obstruction to lifting $\phi$.
\end{rem}

\begin{rem}
\label{rem:fibre_of_map_between_mapping_spaces_of_potential_stages}
In the case $\phi: R^\prime \rightarrow S^\prime$ does lift to a map $R \rightarrow S$, \cref{prop:relation_between_mapping_spaces_between_potential_stages_in_the_multiplicative_case} implies that the fibre $F_{\phi}$ of $\map_{\textnormal{Alg}_{\ekoperad{k}}(\mathcal{M}_{n})}(R, S) \rightarrow \map_{\textnormal{Alg}_{\ekoperad{k}}(\mathcal{M}_{n-1})}(uR, uS)$ can be described as 

\begin{center}
$F_{\phi} \simeq \Omega (\map_{\pi_{0}R} ^{\ekoperad{k}} (\mathbb{L}_{\pi_{0}R / \pi_{0} A}^{\ekoperad{k}}, \Sigma^{n+1} \pi_{0} S[-n])).$
\end{center}
By prestability, this is equivalent to
\[	\map_{\pi_{0}R} ^{\ekoperad{k}} (\mathbb{L}_{\pi_{0}R / \pi_{0} A}^{\ekoperad{k}}, \Sigma^{n} \pi_{0} S[-n]).	\]
\end{rem}

\begin{rem}
By applying \cref{rem:obstruction_to_lifting_morphism_of_potential_stages} to the case of $uR = uS$ and $\phi = id$, we obtain an obstruction to the uniqueness of a lift of a potential $(n-1)$-stage for an $\ekoperad{k}$-algebra, complementing the obstruction to existence of \cref{thm:obstructions_to_lifting_in_multiplicative_case}. To see this, notice that the functors $\textnormal{Alg}_{\ekoperad{k}}(\mathcal{M}_{n}) \rightarrow \textnormal{Alg}_{\ekoperad{k}}(\mathcal{M}_{n-1})$ are conservative by the virtue of \cref{lemma:extension_along_0_equivalence_is_conservative} and so any lift of the identity is necessarily an equivalence. 
\end{rem}

\begin{cor}[Mapping space spectral sequence]
\label{cor:mapping_space_spectral_sequence}
Let $\ccat$ be complete and suppose that $\phi:R \to S$ is a morphism of periodic $\ekoperad{k}$-$A$-algebras in $\ccat$. Then there is a first quadrant spectral sequence with
\[	E_1^{0,0} = \map_{\Alg_{\pi_0A}}(\pi_0R, \pi_0S)	\]
and
\[	E_1^{s,t} = \Ext^{2s-t,s}_{\Mod_{\ekoperad{k}}(\pi_0R)}(\LL^{\ekoperad{k}}_{\pi_0R/\pi_0A}, \pi_0S), \quad t\ge s > 0,	\]
where $\pi_0S$ is given the $\pi_0R$-module structure induced by $\phi$, and converging conditionally to
\[	\pi_{t-s}(\map_{\Alg_{\ekoperad{k}}(\Mod_A(\ccat))}(R, S), \phi),	\]
the homotopy groups of the space of $\ekoperad{k}$-algebra maps from $R$ to $S$, based at $\phi$.
\end{cor}

\begin{proof}
Compare the linear analogue, \cref{cor:mapping_space_spectral_sequence_linear}. Again, this is a case of the Bousfield-Kan spectral sequence \cite{bousfield1972homotopy}, applied to the tower 
\[	\mathrm{maps}_s := \map_{\Alg_{\ekoperad{k}}(\mathcal{M}_s)}(R_{\le s},S_{\le s}),	\]
where each $\mathrm{maps}_s$ is pointed by the image of $\phi$. By completeness of $\ccat$, the limit of the tower is
\[	\map_{\Alg_{\ekoperad{k}}(\ccat)}(R, S).	\]
Then the Bousfield-Kan spectral sequence converges conditionally to the homotopy groups of this space, see \cref{rem:conditional_convergence}.

The $E_1$ term of the spectral sequence is
\[	E_1^{s,t} = \pi_{t-s}\mathrm{fib}(\mathrm{maps}_s \to \mathrm{maps}_{s-1}), \quad t \ge s > 0.	\]
By \cref{rem:fibre_of_map_between_mapping_spaces_of_potential_stages}, we have, for $s > 0$,
\begin{align*}
	E_1^{s,t} &\cong \pi_{t-s}\map_{\pi_{0}R} ^{\ekoperad{k}} (\mathbb{L}_{\pi_{0}R / \pi_{0} A}^{\ekoperad{k}}, \Sigma^s\pi_0S[-s]) \\
	&\cong \pi_{0}\map_{\pi_{0}R} ^{\ekoperad{k}} (\mathbb{L}_{\pi_{0}R / \pi_{0} A}^{\ekoperad{k}}, \Sigma^{2s-t}\pi_0S[-s]) \\
	&\cong \Ext^{2s-t,s}_{\Mod_{\ekoperad{k}}(\pi_0R)}(\LL^{\ekoperad{k}}_{\pi_0R/\pi_0A}, \pi_0S).
\end{align*}
Notice that these spectral sequence terms, which are \textit{a priori} homotopy sets (when $t-s = 0$) or groups (when $t-s = 1$), are in fact always abelian groups. Finally, when $s = 0$, we have
\[	E_1^{0,t} = \pi_t\mathrm{maps}_0 = \pi_t\map_{\Alg_{\pi_0A}}(\pi_0R, \pi_0S).	\]
Here, $\pi_0R$ and $\pi_0S$ are discrete algebras over $\pi_0A$ in $\ccat^\heartsuit$, commutative if $k \ge 2$. In particular, the space of algebra maps is discrete, so the only nonzero term on the $s = 0$ line of the spectral sequence is the pointed set
\[	E_1^{0,0} = \map_{\Alg_{\pi_0A}}(\pi_0R, \pi_0S).	\]
\end{proof}

\section{Moduli of ring spectra} 

In this section we adapt the abstract Goerss-Hopkins theory developed in the previous sections to the particular problem of existence of commutative ring spectra with prescribed homology by specializing to the Grothendieck prestable $\infty$-category of synthetic spectra. 

One way of stating the problem is as follows: given a homotopy commutative ring spectrum $E$ and a commutative algebra $A$ in $E_*E$-comodules, what is the moduli space of \textit{realizations} of $A$ by an $\ekoperad{\infty}$-ring spectrum -- that is, the space of $\ekoperad{\infty}$-ring spectra for which there exists an isomorphism
\[	E_*R \cong A	\]
of $E_*E$-comodule algebras? In particular, one might want to prove that this space is nonempty, meaning that there exists a realization, or contractible, meaning that there exists an essentially unique realization. 

Any obstruction theory of this kind suffers from some natural limitations. Since all of the starting data is in terms of $E$-homology, one will be unable to distinguish a realization of $A$ from its $E$-localization. Even in the best case, then, one will only be able to compute the space of $E$-local realizations. 

Additionally, this approach is only reasonable if we put additional conditions on $E$ which make the categories of $E_*E$-comodules and the $E$-based synthetic spectra well-behaved. One of these conditions is the classical Adams condition, to be reviewed momentarily. To check that a realization is actually produced if all obstructions vanish, we require an additional completeness condition, filling a long-unnoticed gap in the literature.

\begin{defin}
\label{defin:Adams_homology_theory}
An \emph{Adams-type homology theory} is a homotopy commutative ring spectrum $E$ that is a filtered colimit $E \simeq \varinjlim E_\alpha$ of finite spectra $E_{\alpha}$ such that

\begin{enumerate}
\item $E_*E_\alpha$ is a finitely generated, projective $E_*$-module and
\item the K\"{u}nneth map $E^*E_\alpha \to \Hom_{E_*}(E_*E_\alpha, E_*)$ is an isomorphism.
\end{enumerate}
\end{defin}

Originally, this technical condition arose in Adams' blue book \cite{adams1995stable} to prove the universal coefficient theorem and set up the Adams spectral sequence based on $E$. It implies that there is a good algebraic theory of $E$-homology: for example, $(E_*, E_*E)$ is a flat Hopf algebroid, and the $E$-homology $E_*X$ of a spectrum $X$ is an $E_*E$-comodule. 

\begin{example}
Examples of Adams-type theories include the sphere spectrum, Eilenberg-MacLane spectra based on fields, the cobordism spectra $MU$ and $MO$, the $K$-theory spectra $K$ and $KO$, and any Landweber exact homology theory.
\end{example}

In the previous sections, we fixed a graded symmetric monoidal, separated Grothendieck prestable $\infty$-category $\ccat$ and a shift algebra $A \in \textnormal{CAlg}(\ccat)$, and we developed an obstruction theory to existence of periodic $A$-algebras $R$ with prescribed $\pi_{0} R$ as a $\pi_{0} A$-algebra.  To apply these methods to the realization problem for commutative ring spectra, we need to find $\ccat$ and a shift algebra $A$ that satisfy the following conditions:

\begin{enumerate}[label=($\diamondsuit$\arabic*)]
\item there's a symmetric monoidal equivalence $\Mod_{A}^{per}(\ccat) \simeq \spectra_{E}$ between the $\infty$-category of periodic $A$-modules in $\ccat$ and the $\infty$-category of $E$-local spectra,
\item there's a symmetric monoidal equivalence $\Mod_{\pi_{0}A}(\ccat^{\heartsuit}) \simeq \ComodE$ between the abelian category of $\pi_{0}A$-modules in $\ccat^{\heartsuit}$ and the abelian category of $E_{*}E$-comodules,
\item the $\infty$-category $\Mod_{A}$ is generated under colimits by periodic modules and
\item the prestable $\infty$-category $\ccat$ is complete, that is, Postnikov towers in $\ccat$ converge. 
\end{enumerate}
Here, $(\diamondsuit 1)$ is needed to identify the realizations we produce with commutative ring spectra; $(\diamondsuit 2)$ is used to identify the input to the obstruction theory with Andr\'e-Quillen cohomology of comodule algebras; while $(\diamondsuit 3)$ and $(\diamondsuit 4)$ are technical assumptions needed to make sure that, respectively, the obstructions can be computed in the derived category and that if they all vanish then the needed realization exists.

We claim that the $\infty$-category $\synspectra$ of hypercomplete, connective synthetic spectra based on $E$ introduced in \cite{pstrkagowski2018synthetic} and mentioned before in \cref{example:unit_of_synspectra_a_periodicity_algebra}, satisfies the properties $(\diamondsuit 1)-(\diamondsuit 3)$, the shift algebra in this case being given by the monoidal unit.

Note that in this note we will only be considering connective, hypercomplete synthetic spectra based on $E$ in the sense of \cite{pstrkagowski2018synthetic}, and we will simply call them \emph{synthetic spectra}, dropping the two adjectives. As above, we will denote their $\infty$-category by $\synspectra$, with understanding that the Adams-type homology theory is fixed. 

Intuitively, $\synspectra$ is an $\infty$-category of ``resolutions of spectra'' and plays the role of the ``derived $\infty$-category of spectra'' (with respect to the $E$-homology functor). Because of our axiomatic approach, the precise details of the construction of $\synspectra$ are not needed, but for the convenience of the reader we briefly recall them here. 

One says that a finite spectrum $P$ is \emph{$E_{*}$-projective} (or just \emph{projective}) if $E_{*}P$ is projective as an $E_{*}$-module. We endow the $\infty$-category $\spectra_{E}^{fp}$ of finite, $E_{*}$-projective spectra with a Grothendieck topology where covering families consist of single maps which induce a surjection on $E$-homology. 

\begin{defin}
\label{defin:synthetic_spectra}
The $\infty$-category $\synspectra$ of \emph{synthetic spectra} is the $\infty$-category of spherical (product-preserving), hypercomplete sheaves of spaces on $\spectra_{E}^{fp}$.	
\end{defin}
The $\infty$-category of synthetic spectra enjoys the following properties: 

\begin{itemize}
\item as a left exact localization of the $\infty$-category of spherical presheaves at homotopy isomorphisms \cite[2.5]{pstrkagowski2018synthetic}, $\synspectra$ is separated Grothendieck prestable \cite[C.3.2.1]{lurie_spectral_algebraic_geometry},
\item  it has a natural grading induced by the suspension functor on finite $E_{*}$-projective spectra; that is, where $X[1](P):= X(\Sigma^{-1}P)$,
\item it is symmetric monoidal with the tensor product induced by the smash product of finite $E_{*}$-projective spectra, with monoidal unit $\monunit$ given by $\monunit(P) := \map(P, S^{0}_{E})$,
\item the map $\tau :\Sigma \monunit[-1] \to \monunit$ given on sections by the morphism $\Sigma\, \map(\Sigma P, S^0_E) \to \map(P, S^0_E)$ adjoint to the equivalence $\map(\Sigma P, S^{0}_E) \simeq \Omega\, \map(P, S^0_E)$ makes $\monunit$ into a shift algebra \cite[4.61, 4.21]{pstrkagowski2018synthetic}.
\end{itemize}

Thus, as a graded symmetric monoidal, separated Grothendieck prestable $\infty$-category equipped with a choice of a shift algebra, $\synspectra$ is an appropriate context for Goerss-Hopkins theory. We will now verify that it also satisfies $(\diamondsuit 1) - (\diamondsuit 4)$, so that the resulting theory can be applied to the problem of realizing comodule algebras as homology of commutative ring spectra.

\begin{prop}
\label{prop:synthetic_spectra_satisfy_the_conditions_needed_for_gh_theory}
Suppose that $E$ is an Adams-type homology theory. Then the pair $(\synspectra, \monunit)$ satisfies the properties $(\diamondsuit 1)$, $(\diamondsuit 2)$ and $(\diamondsuit 3)$ listed above. 
\end{prop}

\begin{proof}
The property $(\diamondsuit 1)$, namely the equivalence $\Mod_{\monunit}^{per}(\synspectra) \simeq \spectra_{E}$, is \cite[4.36, 5.6]{pstrkagowski2018synthetic}. The property $(\diamondsuit 2)$, that is, $\Mod_{\pi_{0} \monunit}(\synspectra^{\heartsuit}) \simeq \ComodE$, is \cite[4.16]{pstrkagowski2018synthetic}. Lastly, $(\diamondsuit 3)$ follows from the fact that $\synspectra$ is generated under colimits by the synthetic analogues of $E$-localizations of finite projective spectra \cite[4.14, 5.6]{pstrkagowski2018synthetic}, which are periodic by \cite[4.61]{pstrkagowski2018synthetic}. 
\end{proof}

\begin{rem}
\label{rem:synthetic_spectra_not_always_complete}
The Grothendieck prestable $\infty$-category $\synspectra$ is separated, but not necessarily complete, so it may not satisfy $(\diamondsuit 4)$. We will later show, however, that it is complete in what is perhaps the most important case, namely that of Morava $E$-theory, see \cref{thm:synthetic_spectra_based_on_morava_e_theory_complete}.

In the cases where the completeness of $\synspectra$ is not clear, it might be still possible to prove the convergence of the Goerss-Hopkins tower by hand for certain classes of comodules, yielding an effective obstruction theory. An example here is given by $H := H \mathbb{F}_{p}$, the mod $p$ Eilenberg-MacLane spectrum; in \cref{thm:goerss_hopkins_tower_converges_for_bounded_below_hh_comodules} we show that the convergence holds for all $H_{*}H$-comodules which are bounded below.
\end{rem}

We will now show how the properties guaranteed by \cref{prop:synthetic_spectra_satisfy_the_conditions_needed_for_gh_theory} together with the results of previous sections imply an obstruction theory to realizing an algebra in comodules as a homology of a commutative ring spectrum. This is fairly straightforward considering the above discussion, but we make it explicit as a guide to a reader who is perhaps only interested in the realization problem for commutative ring spectra. 

Let $E$ be an Adams-type homology theory such that the $\infty$-category $\synspectra$ is complete, so that $(\diamondsuit 4)$ holds, for example, $E$ can be Morava $E$-theory, see \cref{thm:synthetic_spectra_based_on_morava_e_theory_complete}. Internally to the $\infty$-category $\synspectra$, one constructs a tower of $\infty$-categories of moduli of potential $n$-stages for an $\ekoperad{\infty}$-algebra 

\begin{center}
$\textnormal{Alg}_{\ekoperad{\infty}}(\mathcal{M}_{\infty}) \rightarrow \ldots \rightarrow \textnormal{Alg}_{\ekoperad{\infty}}(\mathcal{M}_{1}) \rightarrow  \textnormal{Alg}_{\ekoperad{\infty}}(\mathcal{M}_{0})$ 
\end{center}
as defined in \cref{defin:a_potential_n_stage_for_an_ek_algebra}. In more detail, $\textnormal{Alg}_{\ekoperad{\infty}}(\mathcal{M}_{n})$ is the $\infty$-category of commutative $\monunitt{n}$-algebras $R$ in synthetic spectra such that $\monunitt{0} \otimes _{\monunitt{n}} R$ is discrete and the connecting functors are given by extension of scalars. By \cref{rem:multiplicative_potential_infty_stages_limit_of_finite_ones}, we have $\textnormal{Alg}_{\ekoperad{\infty}}(\mathcal{M}_{\infty}) \simeq \varprojlim \textnormal{Alg}_{\ekoperad{\infty}}(\mathcal{M}_{n})$.

Then, by the properties $(\diamondsuit 1)$ and $(\diamondsuit 2)$ we have equivalences $\textnormal{Alg}_{\ekoperad{\infty}}(\spectra_{E}) \simeq \textnormal{Alg}_{\ekoperad{\infty}}(\mathcal{M}_{\infty})$ and $\textnormal{Alg}_{\ekoperad{\infty}}(\mathcal{M}_{0}) \simeq \textnormal{Alg}_{\ekoperad{\infty}}(\ComodE)$, so that the top and the bottom of this tower can be identified, respectively, with the $\infty$-category of $E$-local $\ekoperad{\infty}$-ring spectra, and the category of commutative $E_{*}E$-comodule algebras. By \cite[4.22]{pstrkagowski2018synthetic}, the functor $\textnormal{Alg}_{\ekoperad{\infty}}(\mathcal{M}_{\infty}) \rightarrow \textnormal{Alg}_{\ekoperad{\infty}}(\mathcal{M}_{0})$ can be identified with the functor sending an $E$-local commutative ring spectrum $R$ to its homology algebra $E_{*}R$. In other words, the above tower can be considered as a kind of a stratification of the homology functor. 

Now suppose that $A$ is a commutative algebra in $E_{*}E$-comodules, so that it determines an object $A$ in $\textnormal{Alg}_{\ekoperad{\infty}}(\mathcal{M}_{0}) \simeq \textnormal{Alg}_{\ekoperad{\infty}}(\ComodE)$. By the above discussion, the $\infty$-category of realizations of $A$ as the $E$-homology of an $\ekoperad{\infty}$-ring spectrum can be identified with the fibre 

\begin{center}
$\textnormal{Alg}_{\ekoperad{\infty}}(\mathcal{M}_{\infty}) \times _{\textnormal{Alg}_{\ekoperad{\infty}}(\mathcal{M}_{0})} \{ A \} \simeq \varprojlim \textnormal{Alg}_{\ekoperad{\infty}}(\mathcal{M}_{n}) \times _{\textnormal{Alg}_{\ekoperad{\infty}}(\mathcal{M}_{0})} \{ A \}$.
\end{center}
The above equivalence shows that to construct a realization of $A$ it is enough to construct a sequence $R_{n} \in \textnormal{Alg}_{\ekoperad{\infty}}(\mathcal{M}_{n})$ of potential $n$-stages such that $R_{0} \simeq A$ and which is compatible in the sense that $\monunitt{n-1} \otimes _{\monunitt{n}} R_{n} \simeq R_{n-1}$. The results of the previous section give us an obstruction theory to inductively constructing such a sequence, proving the following. 

\begin{thm}[Goerss-Hopkins, {\cite[3.3.5,3.3.7]{moduli_problems_for_structured_ring_spectra}}]
\label{thm:obstructions_to_realizing_an_algebra_in_comodules}
Let $A$ be a commutative algebra in $E_{*}E$-comodules. Then, there exists an inductively defined sequence of obstructions valued in Andr\'{e}-Quillen cohomology groups 

\begin{center}
$\Ext_{\Mod_{A}(\ComodE)}^{n+2, n}(\mathbb{L}_{A / E_{*}}^{\ekoperad{\infty}},  A)$, where $n \geq 1$,
\end{center}
which vanish if and only there exists a realization of $A$, that is, an $\ekoperad{\infty}$-ring spectrum $R$ such that $E_{*}R \simeq A$ as comodule algebras.
\end{thm}

\begin{proof}
Keeping in mind the above discussion, this is just a restatement of \cref{cor:if_c_is_complete_we_have_inductive_obstructions_to_existence_of_periodic_algebras} where $k = \infty$ and the ambient $\infty$-category $\ccat$ is the $\infty$-category $\synspectra$ of synthetic spectra. Note that since we're in the $\ekoperad{\infty}$-case, the $\Ext$-groups are computed in the derived category as above by \cref{rem:obstructions_in_derived_category_in_commutative_case}. 
\end{proof}

Similarly to the obstructions to existence, we also have obstructions to lifting morphisms.

\begin{thm}[Goerss-Hopkins, {\cite[3.3.5]{moduli_problems_for_structured_ring_spectra}}]
\label{thm:obstruction_to_lifting_of_maps_of_potential_stages_for_realization_of_comodule_algebra}
Let $A, B$ be commutative algebras in comodules and $R_{n}, S_{n} \in \textnormal{Alg}_{\ekoperad{\infty}}(\mathcal{M}_{n})$ be potential $n$-stages for realizations of, respectively, $A$ and $B$. Suppose that $\phi: R_{n-1} \rightarrow S_{n-1}$ is a morphism of the corresponding potential $(n-1)$-stages. 

Then, there exists an obstruction in $\Ext_{\Mod_{A}(\ComodE)}^{n+1, n}(\mathbb{L}_{A / E_{*}}^{\ekoperad{\infty}}, B)$ which vanishes if and only if $\phi$ can be lifted to a morphism $R_{n} \rightarrow S_{n}$. Moreover, if such a lift exists, then we have a fibre sequence

\begin{center}
$\map_{\dcat(\Mod_{A}(\ComodE))}(\mathbb{L}^{\ekoperad{\infty}}_{A / E_{*}}, \Sigma^{n} B[-n]) \rightarrow \map(R_{n}, S_{n}) \rightarrow \map(R_{n-1}, S_{n-1})$ 
\end{center}
involving, respectively, the Andr\'{e}-Quillen cohomology space and the mapping spaces between potential $n$- and $(n-1)$-stages for an $\ekoperad{\infty}$-algebra.
\end{thm}

\begin{proof}
Similarly to \cref{thm:obstructions_to_realizing_an_algebra_in_comodules}, this is an appropriate restatement of \cref{prop:relation_between_mapping_spaces_between_potential_stages_in_the_multiplicative_case}, or more precisely its form given by \cref{rem:obstruction_to_lifting_morphism_of_potential_stages} and \cref{rem:fibre_of_map_between_mapping_spaces_of_potential_stages}. 
\end{proof}

\begin{cor}[{\cite[4.3]{moduli_spaces_of_commutative_ring_spectra}}]
\label{cor:mapping_space_spectral_sequence_ring_spectra}
Let $\phi:R \to S$ be a morphism of $E$-local $\ekoperad{\infty}$-ring spectra. Then there is a first quadrant spectral sequence with
\[	E_1^{0,0} = \map_{\CAlg(\ComodE)}(E_*R, E_*S)	\]
and
\[	E_1^{s,t} = \Ext^{2s-t,s}_{\Mod_{E_*R}(\ComodE)}( \mathbb{L}^{\ekoperad{\infty}}_{E_*R/E_*}, E_*S), \quad t \ge s > 0,	\]
where $E_*S$ is given an $E_*R$-module structure via $\phi$, and converging conditionally to
\[	\pi_{t-s}(\map_{\ekoperad{\infty}}(R, S), \phi),	\]
the homotopy groups of the space of $\ekoperad{\infty}$-maps from $R$ to $S$, based at $\phi$.
\end{cor}

\begin{proof}
This is an immediate application of \cref{cor:mapping_space_spectral_sequence} to the present case.
\end{proof}

\begin{rem}
This spectral sequence of \cref{cor:mapping_space_spectral_sequence_ring_spectra} agrees with that of Goerss and Hopkins only after suitable reindexing: our $E_1^{s,t}$ is isomorphic to their $E_2^{2s-t,s}$. We believe that the two spectral sequences are related by the Deligne's d\'ecalage construction of \cite[Proposition 6.3]{levine2015adams}. 

The discrepancy comes from the different methods of construction. The Goerss-Hopkins mapping space spectral sequence is constructed without using the decomposition of the moduli space into moduli of potential $n$-stages \cite{moduli_spaces_of_commutative_ring_spectra}. Instead, Goerss and Hopkins resolve $R$ and $S$ by suitably free algebras, producing a cosimplicial object totalizing to the mapping space. The $E_1$ page of their spectral sequence depends on this choice of resolution, but the $E_2$ page is a homotopy invariant of the ring spectrum.

Our mapping space is the Bousfield-Kan spectral sequence for the tower of mapping spaces of potential $n$-stages; since the notion of potential $n$-stage is homotopy-invariant, the resulting spectral sequence is well-defined at the $E_1$ page. 
\end{rem}

\begin{rem}
We did not use it, but the $\infty$-category $\synspectra$ of synthetic spectra is in fact equivalent to the underlying $\infty$-category of the semi-model category used by Goerss and Hopkins, and using this fact one can translate between our account and the one given in \cite{moduli_spaces_of_commutative_ring_spectra}, \cite{moduli_problems_for_structured_ring_spectra}. We give a short sketch of this equivalence. 

It is a folklore result, which unfortunately doesn't seem to be written down, that the underlying $\infty$-category of the Bousfield model structure with respect to a set of compact generators is equivalent to the $\infty$-category of spherical presheaves on those generators. In the case of the particular Bousfield model structure appearing in the work of Goerss and Hopkins, for a maximal choice of generators, this translates into an equivalence between this underlying $\infty$-category and the $\infty$-category $P_{\Sigma}(\spectra_{E}^{fp})$ of spherical presheaves on finite projective spectra. 

Since the semi-model structure that Goerss and Hopkins work with is defined as a Bousfield localization, it follows that its underlying $\infty$-category is equivalent to a localization of $P_{\Sigma}(\spectra_{E}^{fp})$; by definition, to the localization along $E_{*}$-quasi-isomorphisms. This localization is precisely the $\infty$-category $\synspectra$, see \cite[5.4]{pstrkagowski2018synthetic}.

\end{rem}

\begin{rem}
\label{rem:goerss_hopkins_issue_of_completeness}
Note that to obtain our obstruction theory we needed to make an additional assumption on top of $E$ being Adams, namely that the $\infty$-category $\synspectra$ of synthetic spectra based on $E$ is complete, in other words, that Postnikov towers converge. This assumption does not appear in the original references \cite{moduli_spaces_of_commutative_ring_spectra}, \cite{moduli_problems_for_structured_ring_spectra}, but we believe this is in fact a mistake. 

The assumption of convergence of Postnikov towers, which we denoted above as $(\diamondsuit 4)$, is only needed to obtain the identification $\textnormal{Alg}_{\ekoperad{\infty}}(\mathcal{M}_{\infty}) \simeq \varprojlim \textnormal{Alg}_{\ekoperad{\infty}}(\mathcal{M}_{n})$. Thus, without it, we can still construct the tower of potential $n$-stages with all the required properties, but we can't deduce that a compatible tower of potential $n$-stages is necessarily induced by a realization in $E$-local ring spectra. 

The proof of the corresponding fact in \cite{moduli_spaces_of_commutative_ring_spectra}, \cite{moduli_problems_for_structured_ring_spectra} is a reference to \cite[4.6]{dwyer_kan_classification_theorem}, a theorem which in modern language can be interpreted as an explicit point-set proof of the fact that Postnikov towers converge in the $\infty$-category of spaces. This, however, does not seem to immediately apply to the semi-model category that Goerss and Hopkins are working with. 

The relevant semi-model category is a Bousfield localization of a model structure on the category of simplicial spectra due to Bousfield, and hence the Postnikov towers can be computed in the latter model structure \cite[2.5.6]{moduli_problems_for_structured_ring_spectra}. There, we have an explicit model for the $n$-th Postnikov stage given by attaching cells, see \cite[3.11]{moduli_spaces_of_commutative_ring_spectra}, and using this model one sees that attaching an $n$-cells does not change the lower homotopy groups. Using this, one can show that Postnikov stages do converge in the Bousfield model category. 

This does not, however, imply a corresponding statement in the localized semi-model category. To see this, note that to obtain a homotopy invariant statement about mapping spaces in this category, we need to perform a fibrant-cofibrant replacement with respect to the localized model structure, and the relevant homotopy groups are not invariant under this replacement, only homology is. Thus, as we attach cells in the semi-model category and follow it by the fibrant-cofibrant replacement, it is possible that the lower homotopy groups change as well, and so Postnikov convergence is not obvious. This, we believe, is the root of the mistake. 

Note that perhaps the most spectacular application of Goerss-Hopkins obstruction theory is the functorial $\ekoperad{\infty}$-ring structure on Morava $E$-theory, and this application is in fact not threatened by this error, as we show below in \cref{thm:synthetic_spectra_based_on_morava_e_theory_complete} that Postnikov towers in synthetic spectra based on Morava $E$-theory do converge. 
\end{rem}

\section{Morava \texorpdfstring{$E$}{E}-theory}

This section is devoted to the classical Goerss-Hopkins obstruction theory in the particular case of Morava $E$-theory. We first show that the $\infty$-category of synthetic spectra based on Morava $E$-theory is complete, a technical assumption needed for the obstruction theory to work. Then, we review the calculation of the relevant obstruction groups which implies that Morava $E$-theory carries a canonical structure of an $\ekoperad{\infty}$-ring spectrum.

Let $k$ be a perfect field of positive characteristic and let $\Gamma$ a formal group law over $k$ of finite height $n$. By results of Lubin and Tate, there exists a universal deformation of $\Gamma$, defined over the ring $W(k)[[u_{1}, \ldots, u_{n-1}]]$ of power series in the Witt vectors over $k$ \cite{lubin1966j}.  

The \emph{Morava $E$-theory} or \emph{Lubin-Tate theory} associated to the pair $(k, \Gamma)$ is the 2-periodic Landweber exact multiplicative homology theory with $E(k, \Gamma)_{*} \simeq W(k)[[u_{1}, \ldots, u_{n-1}]][u^{+-1}]$ whose corresponding formal group law is the universal deformation of $\Gamma$ \cite{rezk1998notes}. 

The homology theory $E(k, \Gamma)$ depends on the height $n$ and characteristic $p$, but dependance on the pair $(k, \Gamma)$ is relatively mild. More precisely, by classical results of Hovey and Strickland, $E(k, \Gamma)$ is Bousfield equivalent to the Johnson-Wilson theory $E(n)$, and similarly, the category of $E(k, \Gamma)_{*}E(k, \Gamma)$-comodules is equivalent to the category of $E(n)_{*}E(n)$-comodules \cite{hovey2003comodules}. For this reason, we will omit $(k, \Gamma)$ from the notation when convenient, writing only $E$.

We will first show that the $\infty$-category $\synspectra$ of synthetic spectra based on $E$ is complete, that is, that the Postnikov towers converge. This fact is related to another extraordinary property of Morava $E$-theory, namely that any $E$-local spectrum is $E$-nilpotent, so that the $E$-based Adams spectral sequence always converges \cite[5.3]{hovey1999invertible}. Our proof is similar, utilizing the existence of the Hopkins-Ravenel finite spectrum.

\begin{lemma}
\label{lemma:for_morava_e_theory_no_ext_into_e_implies_finite_injective_dimension}
Let $E$ be Morava $E$-theory and suppose that $M$ is a dualizable $E_{*}E$-comodule such that $\Ext^{s, t}_{E_{*}E}(M, E_{*}) = 0$ for $s > n^{2} + n$. Then, $\Ext^{s, t}_{E_{*}E}(M, N) = 0$ for $s > n^{2} + n$ and all comodules $N$. 
\end{lemma}

\begin{proof}
 Let us say that a comodule $N$ is \emph{good} if it satisfies the condition given above, by assumption we know that the shifts of $E_{*}$ are good. Similarly, the long exact sequence of $\Ext$-groups implies that if $0 \rightarrow N \rightarrow N^{\prime} \rightarrow N^{\prime \prime} \rightarrow 0$ is short exact and $N$ is good, then $N^{\prime}$ is good if and only if $N^{\prime \prime}$ is. We deduce that any comodule that admits a Landweber filtration in the sense of \cite{hovey2003comodules} is good and, hence, that all finitely generated comodules are. 

Since $E$ is Adams, any $E_{*}E$-comodule is a filtered colimit of finitely generated ones. Because $M$ is dualizable, $\Ext^{s, t}(M, N)$ can be computed using the cobar complex and so commutes with filtered colimits in $N$. Thus, we deduce that all comodules are good, which is what we wanted to show. 
\end{proof}

\begin{lemma}
\label{lemma:any_collection_of_fp_spectra_which_surjects_on_all_others_generates_synthetic_spectra}
Let $E$ be an Adams-type homology theory and $\ccat$ be a collection of finite projective spectra such that for any $P \in \spectra_{E}^{fp}$ there exists an $S \in \ccat$ and an $E_{*}$-surjection $S \rightarrow P$. Then, the synthetic analogues $\nu S_{E}$ of $E$-localizations of $S \in \ccat$ generate the $\infty$-category $\synspectra$ of synthetic spectra based on $E$ under colimits. 
\end{lemma}

\begin{proof}
By construction, $\synspectra$ is generated under colimits by the synthetic analogues $\nu P_{E}$ of $E$-localizations of finite projective spectra, where we're implictly using that $\nu P_{E}$ is the hypercompletion of the non-hypercomplete synthetic spectrum $\nu P$ \cite[5.6]{pstrkagowski2018synthetic}. Thus, it is enough to write down any such $\nu P_{E}$ as a colimit of $\nu S_{E}$ as above. 

Now, by assumption for any finite projective $P$ we can find an $E_{*}$-surjection $S \rightarrow P$ where $S \in \ccat$. Proceeding inductively using \cite[A.22]{pstrkagowski2018synthetic} we can extend this map to a hypercover; that is, a semisimplicial augmented object $U: \Delta^{op}_{s, +} \rightarrow \spectra_{E}^{fp}$ in finite projective spectra such that all of the maps $U_{n} \rightarrow M_{n}(U)$ are $E_{*}$-surjective. Moreover, by our assumption we can guarantee that $U_{n} \in \ccat$ for $n \geq 0$. 

By \cite[A.23]{pstrkagowski2018synthetic}, in the hypercomplete setting the Yoneda embedding takes hypercovers to colimit diagrams. We deduce that $\varinjlim \nu (U_{n})_{E} \simeq \nu P_{E}$, which is what we wanted to show. 
\end{proof}

\begin{lemma}
\label{lemma:for_morava_e_theory_any_fp_covered_by_one_with_homology_of_finite_homological_dimension}
Let $E$ be Morava $E$-theory. Then, for any finite projective spectrum $P$ there exists a surjection $S \rightarrow P$ from a finite projective $S$ such that $\Ext_{E_{*}E}^{s, t}(E_{*}S, N) = 0$ for $s > n^{2} + n$ and all comodules $N$. 
\end{lemma}

\begin{proof}
We first claim such a spectrum exists when $P$ is the sphere $S^{0}$. Recall that Hopkins and Ravenel construct a finite spectrum $T$ with torsion-free homology such that $\Ext^{s}_{E_{*}E}(E_{*}, E_{*}T)$ vanishes for $s > n^{2}+n$ \cite[8.3]{ravenel2016nilpotence}. Note that torsion-freeness implies that $E_{*}T$ is dualizable, since $E_{*}T  \simeq E_{*} \otimes _{MU_{*}} MU_{*}T$ and $MU_{*}T$ is known to be free, so that $T$ is finite projective. 

Taking the Spanier-Whitehead dual $S := F(T, S^{0})$ we see that $S$ is a finite projective spectrum such that $\Ext^{s, t}_{E_{*}E}(E_{*}S, E_{*}) = 0$ for $s > n^{2}+n$ and so $\Ext^{s, t}_{E_{*}E}(E_{*}S, N) = 0$ for $s > n^{2} + n$ and all $N$ by \cref{lemma:for_morava_e_theory_no_ext_into_e_implies_finite_injective_dimension}. Taking suspensions, we can guarantee that the top cell in a minimal cell structure on $S$ is in degree zero. Then, the projection $S \rightarrow S^{0}$ onto the top cell is the needed $E_{*}$-surjection onto $S^{0}$.

Now suppose that $P$ is an arbitrary finite projective. If $S \rightarrow S^{0}$ is the $E_{*}$-surjection constructed above, then $S \wedge P \rightarrow P$ is also $E_{*}$-surjective. Finite projective spectra are closed under the smash product \cite[3.18]{pstrkagowski2018synthetic}, so $S \wedge P$ is also finite projective. Finally, we have 
\begin{center}
$\Ext^{s, t}_{E_{*}E}(E_{*}(S \wedge P), N) \simeq \Ext_{E_{*}E}^{s, t}(E_{*}S \otimes _{E_{*}} E_{*}P, N) \simeq \Ext_{E_{*}E}^{s, t}(E_{*}S, \Hom_{E_{*}}(E_{*}P, E_{*}) \otimes _{E_{*}} N)$,
\end{center}
so that $S \wedge P$ also has the required property, ending the argument. 
\end{proof}

\begin{thm}
\label{thm:synthetic_spectra_based_on_morava_e_theory_complete}
Let $E$ be the Morava theory at a fixed prime and height. Then, the Grothendieck prestable $\infty$-category $\synspectra$ of hypercomplete, connective synthetic spectra based on $E$ is complete. 
\end{thm}

\begin{proof}
It is enough to show that for any $X \in \synspectra$ and any $m \geq 0$, the projection $\varprojlim X_{\leq i} \rightarrow X_{\leq k}$ is an $m$-equivalence for all $k$ large enough \cite[5.5.6.27]{lurie_higher_topos_theory}. To do this, it is enough to verify that for all $k$ large enough, the collection of $Y$ such that 

\begin{center}
$\map(Y, \varprojlim X_{\leq i}) \rightarrow \map(Y, X_{\leq k})$
\end{center}
is an $(m+1)$-equivalence generates $\synspectra$ under colimits. 

We claim that when $k > m + n^{2}+n$, that collection of $Y$ satisfying the above condition contains the synthetic analogues $\nu P_{E}$ of $E$-local finite projectives such that $\Ext_{E_{*}E}^{s, t}(E_{*}P, N) = 0$ for all $s > n^{2} + n$ and all comodules $N$. Since the latter generate $\synspectra$ under colimits by a combination of \cref{lemma:for_morava_e_theory_any_fp_covered_by_one_with_homology_of_finite_homological_dimension} and \cref{lemma:any_collection_of_fp_spectra_which_surjects_on_all_others_generates_synthetic_spectra}, this will finish the argument. 

Let $X[i]$ denote the fibre of $X_{\leq i} \rightarrow X_{\leq i-1}$, this is a suspension of a discrete synthetic spectrum and so is canonically a $\monunit_{\leq 0}$-module. It follows that we have equivalences

\begin{center}
$\map_{\synspectra}(\nu P_{E}, X[i]) \simeq \map_{\Mod_{\monunitt{0}}(\synspectra)}(\monunitt{0} \otimes \nu P_{E}, X[i]) \simeq \map_{\dcat(\ComodE)}(E_{*}P, \Sigma^{i} \pi_{i} X)$,
\end{center}
where the second one follows from $\Mod_{\monunitt{0}}(\synspectra) \simeq \dcat(\ComodE)_{\leq 0}$, which we have shown in \cref{prop:synthetic_spectra_satisfy_the_conditions_needed_for_gh_theory}. 

Since $\pi_{j} \map_{\dcat(\ComodE)}(E_{*}P, \Sigma^{i} \pi_{i} X) \simeq \Ext^{i-j}_{E_{*}E}(E_{*}P, \pi_{i} X)$ and we assume that this vanishes when $i-j > n^{2}+n$, we see that this space grows highly connected as $i$ goes to infinity. Thus, we deduce that the homotopy groups $\pi_{j} \map_{\synspectra}(\nu P_{E}, X_{\leq i})$ stabilize for $i > n^{2}+n+j$, and an application of the Milnor exact sequence then yields the needed result. 
\end{proof}

\begin{rem}
\label{rem:abelian_category_generated_by_objects_of_finite_homological_dimension_has_left_complete_derived_category}
The same argument as the one used in the proof of \cref{thm:synthetic_spectra_based_on_morava_e_theory_complete} would show that the derived category $\dcat(\acat)$ of a Grothendieck abelian category $\acat$ generated under colimits by objects of fixed finite homological dimension (intuitively, $\acat$ is ``virtually of finite homological dimension'') is left complete. This statement appears to be foklore \cite[1.2.10]{drinfeld2013some}. 
\end{rem}

For the convenience of the reader, we will now recall the proof of the celebrated Goerss-Hopkins-Miller theorem, which states that Morava $E$-theory admits a unique $\ekoperad{\infty}$-ring structure. This requires us to show that the obstruction groups appearing in the statement of \cref{thm:obstructions_to_realizing_an_algebra_in_comodules} vanish, which is a non-trivial result in its own right.

We start with a short technical lemma about the cotangent complex, which will help us reduce the relevant Andr\'e-Quillen cohomology groups to a simpler form.

\begin{lemma}[General base-change for the cotangent complex] 
\label{lemma:general_base_change_for_cotangent_complex}
Let $f: \ccat \rightarrow \dcat$ be a symmetric monoidal, cocontinuous functor between presentably symmetric monoidal stable $\infty$-categories. Then, for any $A \in \textnormal{Alg}_{\ekoperad{\infty}}(\ccat)$ we have $\mathbb{L}^{\ekoperad{\infty}}_{f(A)} \simeq f_{*}(\mathbb{L}_{A}^{\ekoperad{\infty}})$, where $f_{*}: \Mod_{A} \rightarrow \Mod_{f(A)}$ is the induced functor between module $\infty$-categories. 
\end{lemma}

\begin{proof}
Since $f$ is a cocontinuous functor between presentable, stable $\infty$-categories, it has a right adjoint $g: \dcat \rightarrow \ccat$ which for formal reasons has a structure of a lax symmetric monoidal functor. It follows that both functors take algebras to algebras and one can verify that this induces another adjunction $f \dashv g: \textnormal{Alg}_{\ekoperad{\infty}}(\ccat) \leftrightarrows \textnormal{Alg}_{\ekoperad{\infty}}(\dcat)$ which we denote with the same letters. 

If $A \in \textnormal{Alg}_{\ekoperad{\infty}}(\ccat)$, then the induced functor $f: \textnormal{Alg}_{\ekoperad{\infty}}(\ccat) _{/A} \rightarrow \textnormal{Alg}_{\ekoperad{\infty}}(\dcat) _{/f(A)}$ is also cocontinuous, since colimits in overcategories are computed objectwise. It follows that we have a commutative diagram 

\begin{center}
	\begin{tikzpicture}
		\node (TL) at (0, 1.3) {$ \textnormal{Alg}_{\ekoperad{\infty}}(\ccat) _{/A} $};
		\node (TR) at (4, 1.3) {$ \spectra(\textnormal{Alg}_{\ekoperad{k}}(\ccat) _{/A}) $};
		\node (BL) at (0, 0) {$ \textnormal{Alg}_{\ekoperad{\infty}}(\dcat) _{/f(A)} $};
		\node (BR) at (4, 0) {$ \spectra(\textnormal{Alg}_{\ekoperad{\infty}}(\dcat) _{/f(A)}) $};
		
		\draw [->] (TL) to node[auto] {$ \Sigma^{\infty}_{+} $} (TR);
		\draw [->] (TL) to (BL);
		\draw [->] (TR) to (BR);
		\draw [->] (BL) to node[below] {$ \Sigma^{\infty}_{+} $} (BR);
	\end{tikzpicture}
\end{center}
of presentable $\infty$-categories and left adjoints induced by stabilization. By \cite[7.3.4.13]{higher_algebra}, the right vertical arrow can be identified with $f_{*}: \Mod_{A} \rightarrow \Mod_{f(A)}$. The statement is then immediate from the fact that $\mathbb{L}^{\ekoperad{\infty}}_{A} := \Sigma^{\infty}_{+} A$, and likewise for $f(A)$, by definition of the cotangent complex \cite[7.3.2.14]{higher_algebra}. 
\end{proof}

To state the theorem, we must introduce the profinite group $\Aut(k,\Gamma)$. An element of $\Aut(k,\Gamma)$ is a commutative square
\[	\xymatrix{ \Gamma \ar[d] \ar[r]^{\phi} & \Gamma \ar[d] \\ \Spec k \ar[r]_{i^*} & \Spec k }	\]
where $i$ is an automorphism of $k$, $\phi$ is a formal group isomorphism, and both vertical maps are the natural structure map of $\Gamma$.

\begin{thm}[Goerss-Hopkins-Miller]
\label{thm:goerss_hopkins_miller}
Let $E = E(k, \Gamma)$ be the Morava $E$-theory spectrum associated to a perfect field $k$ of characteristic $p$ and a formal group $\Gamma$ of finite height, and let $\mathcal{M}_{\ekoperad{\infty}}(E)$ denote the $\infty$-groupoid whose objects are $\ekoperad{\infty}$-ring spectra whose underlying homotopy associative ring spectrum is equivalent to $E$ and whose morphisms are $\ekoperad{\infty}$-ring equivalences.  

Then, there's an equivalence $\mathcal{M}_{\ekoperad{\infty}}(E) \simeq B\Aut(k,\Gamma)$; in other words, there's a unique $\ekoperad{\infty}$-ring spectrum equivalent to $E$ and its automorphism group is $Aut(k, \Gamma)$. 
\end{thm}

\begin{proof}
We claim that a homotopy associative ring spectrum $A$ is equivalent to $E$ if and only if $E_{*}A \simeq E_{*}E$ as comodule algebras, so that we have an equivalence 

\begin{center}
$\mathcal{M}_{\ekoperad{\infty}}(E) \simeq \textnormal{Alg}^{E_{*}E}_{\ekoperad{\infty}}(\spectra_{E})$,
\end{center}
where the former is as in the statement and the latter is the $\infty$-groupoid of realizations of $E_*E$ in $E$-local $\ekoperad{\infty}$-ring spectra. One direction is easy: if $A$ is equivalent to $E$, then $E_*A \simeq E_*E$. Let us instead assume that we have $E_{*}E \simeq E_{*}A$ as comodule algebras. This implies that $E_{*}A$ is flat over $E_{*}$, so that the $E_{2}$-page of the universal coefficient spectral sequence

\begin{center}
$\Ext^{s, t}_{E_{*}}(E_{*}A, E_{*}) \Rightarrow [A, E]_{t-s}$
\end{center}
vanishes for $s > 0$ by \cite[15.6]{rezk1998notes}, forcing the spectral sequence to collapse. We deduce that the chosen isomorphism in $\Hom_{E_{*}E}(E_{*}A, E_{*}E) \simeq \Hom_{E_{*}}(E_{*}A, E_{*})$ descends to an equivalence $A \simeq E$. Since the chosen isomorphism $E_{*}A \simeq E_{*}E$ was an isomorphism of algebras, $A \simeq E$ is an equivalence of homotopy associative rings, which is what we wanted to show. 

We now consider the problem of realizing the $E_*E$-comodule algebra $E_*E$ as the homology of an $E$-local $\ekoperad{\infty}$-ring spectrum. As $E$ satisfies the Adams condition and $\synspectra$ is complete, we can study this problem using the obstruction theory of \cref{thm:obstructions_to_realizing_an_algebra_in_comodules} and \cref{thm:obstruction_to_lifting_of_maps_of_potential_stages_for_realization_of_comodule_algebra}.

The obstructions to existence and uniqueness live in groups of the form
\[	\Ext_{\Mod_{E_*E}(\dcat(\Comod_{E_*E}))}^{s,t}(\LL_{E_*E/E_*}^{\ComodE}, E_*E),	\]
where we dropped $\ekoperad{\infty}$ from the notation for the cotangent complex in favor of emphasizing that it is internal to the category of comodules. We claim that all of these obstruction groups vanish; our first step is to reduce this calculation to one internal to the world of $E_{*}$-modules. 

The forgetful functor $U:\ComodE \rightarrow \Mod_{E_*}^{\heartsuit}$  from comodules into discrete $E_{*}$-modules admits a right adjoint given by the cofree comodule construction $M \mapsto E_{*}E \otimes _{E*} M$. One sees that if $M$ is an $E_*E$-module, then $E_*E \otimes_{E_*} M$ is a module over $E_*E$ in comodules, inducing an adjunction $E \dashv E_{*}E \otimes _{*} (-): \Mod_{E_{*}E} (\ComodE) \leftrightarrows \Mod_{E_{*}E}^{\heartsuit}$ between categories of modules. It is easy to see that both adjoints are exact, using flatness of $E_*E$ over $E_*$, so that this adjunction passes to one between derived $\infty$-categories.

Using the derived adjunction, we see that
\begin{center}
$\Ext_{\Mod_{E_*E}(\dcat(\Comod_{E_*E}))}^{s,*}(\LL_{E_*E/E_*}^{\ComodE}, E_*E) \simeq \Ext_{\Mod_{E_*E}}^{s,*}(U_*(\LL_{E_*E/E_*}^{\ComodE}), E_*)$
\end{center}
and further by the base change formula of \cref{lemma:general_base_change_for_cotangent_complex} that 

\begin{center}
$ \Ext_{\Mod_{E_*E}}^{s,*}(U_*(\LL_{E_*E/E_*}^{\ComodE}), E_*) \simeq \Ext_{\Mod_{E_*E}}^{s,*}(\LL_{E_*E/E_*}^{\Mod_{E_*}^{\heartsuit}}, E_*)$,
\end{center}
where we've used that the forgetful functor is cocontinuous and symmetric monoidal. 

Having reduced to the case of modules, let us write $\LL_{E_*E / E_*}^{\ekoperad{\infty}} := \LL_{E_*E/E_*}^{\Mod_{E_*}^{\heartsuit}}$; that is, the cotangent complex is now implicitly taken in $E_{*}$-modules. The ring $E_*$ is complete with respect to a maximal ideal $\mathfrak{m}$ and filtering $E_*$ by powers of $\mathfrak{m}$ gives a spectral sequence
\[	E_2^{p,q} = \Ext^{p,*}_{E_*}(\LL_{E_*E/E_*}^{\ekoperad{\infty}}, \mathfrak{m}^q/\mathfrak{m}^{q+1}) \Rightarrow \Ext^{p+q,*}_{E_*}(\LL_{E_*E/E_*}^{\ekoperad{\infty}}, E_*).	\]

The $E_2$ page of this spectral sequence is equivalent to
\[	\Ext^{p,*}_{E_*/\mathfrak{m}}(\LL_{E_*E/E_*}^{\ekoperad{\infty}} \otimes^\LL_{E_*} E_*/\mathfrak{m}, \mathfrak{m}^q/\mathfrak{m}^{q+1}),	\]
and so it suffices to prove the vanishing of
\[	\LL_{E_*E/E_*}^{\ekoperad{\infty}} \otimes^\LL_{E_*} E_*/\mathfrak{m}.	\]
Since $E_*E$ is flat over $E_*$, this is equivalent to
\[	\LL_{(E_*E/\mathfrak{m})/(E_*/\mathfrak{m})}^{\ekoperad{\infty}} \simeq E_* \otimes_{E_0} \LL_{(E_0E/\mathfrak{m})/(E_0/\mathfrak{m})}^{\ekoperad{\infty}}.	\]

By construction of Morava $E$-theory, we have $E_0/\mathfrak{m} = k$, a perfect field of characteristic $p$, and it is well-known that 
\[	E_0E/\mathfrak{m} \cong \mathrm{maps}_{cts}(\mathrm{Aut}(k,\Gamma), k),	\]
where the latter is the ring of continuous maps from the profinite group $\mathrm{Aut}(k,\Gamma)$ to the discrete field $k$, see \cite{strickland_gross_hopkins}. In particular, $E_0E/\mathfrak{m}$ is perfect, which implies that the $\ekoperad{\infty}$ cotangent complex $\LL_{(E_0E/\mathfrak{m})/(E_0/\mathfrak{m})}^{\ekoperad{\infty}}$ vanishes \cite[Lemma 5.2.8]{lurie_elliptic_cohomology_two}.

This implies that the obstruction groups of \cref{thm:obstructions_to_realizing_an_algebra_in_comodules}, which are the Andr\'e-Quillen cohomology groups
\[	 \Ext_{\Mod_{E_*E}(\dcat(\Comod_{E_*E}))}^{t+2,t}(\LL_{E_*E/E_*}^{\ComodE}, E_*E),	\]
vanish, so that a realization of $E_{*}E$ exists. Moreover, the obstructions to lifting morphisms of \cref{thm:obstruction_to_lifting_of_maps_of_potential_stages_for_realization_of_comodule_algebra} also vanish and the functors $\textnormal{Alg}_{\ekoperad{\infty}}(\mathcal{M}^{E_{*}E}_{t+1}) \to \textnormal{Alg}_{\ekoperad{\infty}}(\mathcal{M}^{E_{*}E}_{t}) $ between the $\infty$-categories of potential $t$-stages for $E_{*}E$ are equivalences. Passing to the limit, we deduce that the same is true for the functor

\begin{center}
$\textnormal{Alg}^{E_{*}E}_{\ekoperad{\infty}}(\spectra_{E}) \rightarrow B\textnormal{Aut}_{\textnormal{CAlg}(\ComodE)}(E_{*}E)$
\end{center}
between realizations of $E_{*}E$ in $E$-local aspectra and the category of comodule algebras isomorphic to $E_{*}E$. The latter is well-known to be equivalent to $B\textnormal{Aut}(k, \Gamma)$, see \cite[\S17]{rezk1998notes}, ending the argument.
\end{proof}

\begin{rem}
Using arguments analogous to the ones appearing in the proof of \cref{thm:goerss_hopkins_miller}, one can show more generally that if $E(k_{1}, \Gamma_{1})$ and $E(k_{2}, \Gamma_{2})$ are two Morava $E$-theory spectra with their unique $\ekoperad{\infty}$-ring structures, then 

\begin{center}
$\map_{\textnormal{Alg}_{\ekoperad{\infty}}(\spectra)}(E(k_{1}, \Gamma_{1}), E(k_{2}, \Gamma_{2})) \simeq \map_{FGL}((k_{1}, \Gamma_{1}), (k_{2}, \Gamma_{2}))$,
\end{center}
where the latter is the set of homomorphisms in the category of rings equipped with a choice of a formal group law. In particular, the space of $\ekoperad{\infty}$-maps is component-wise contractible. 
\end{rem}

\begin{rem}
Lurie has recently given a very different proof of the existence of an $\ekoperad{\infty}$-algebra structure on $E(k,\Gamma)$, by realizing it as a solution to a moduli problem in spectral algebraic geometry \cite{lurie_elliptic_cohomology_two}. In fact, Lurie's construction is much more general and produces commutative ring spectra associated to deformation problems for a wide class of rings and p-divisible groups.
\end{rem}

\begin{rem}
One can imagine a different obstruction-theoretic proof of the uniqueness of the $\ekoperad{\infty}$-ring structure on $E(k, \Gamma)$ by using Andr\'e-Quillen cohomology of algebras with power operations, analogous to the one using $\theta$-algebras in \cite{moduli_problems_for_structured_ring_spectra}. 

The relevant power operations on the $E(k,\Gamma)$-homology of an $\ekoperad{\infty}$-algebra are given by the Rezk monad $\mathbb{T}$ \cite{rezk2009congruence}; as the free $\mathbb{T}$-algebra on one generator is a polynomial algebra \cite{strickland1998morava}, the $\mathbb{T}$-algebra cotangent complex agrees with the ordinary commutative algebra cotangent complex, which vanishes by formal \'etaleness. An odd feature of this argument is that one needs to assume that an $\ekoperad{\infty}$-ring structure (or at least an $H_\infty$-structure) exists on $E(k,\Gamma)$ to produce the obstruction theory in the first place.
\end{rem}

\section{Toda obstruction theory}
Let $H$ be the Eilenberg-MacLane spectrum $H\mathbb{F}_p$ at a fixed prime and $A$ the associated Steenrod algebra. In \cite{toda1971spectra}, Toda gave an obstruction theory for realizing a bounded below module over the Steenrod algebra as the cohomology of a spectrum. 

\begin{thm}[Toda's obstruction theory, {\cite[3.2]{bhattacharya2016class}}]
\label{thm:toda_obstruction_theory}
Let $M$ be a bounded below $A$-module. Then, there exists a sequence of inductively defined obstructions
\begin{center}
$\theta_{n} \in \Ext_{A}^{n+2, n}(M, M)$, where $n \geq 1$
\end{center}
which vanish if and only if there exists a spectrum $X$ such that $M \simeq H^{*}X$ as $A$-modules.
\end{thm}

In this section, we will prove a homological variant of Toda's theorem as a consequence of the Goerss-Hopkins obstruction theory developed in this note. In fact, our approach yields a slight strengthening of Toda's result, as it produces not just obstructions to realizations but an inductive decomposition of the moduli space of realizations. It is also a particularly interesting strength test for our theory, for the following reason.

In the case of Morava $E$-theory, considered in \cref{thm:synthetic_spectra_based_on_morava_e_theory_complete}, the Goerss-Hopkins tower
\begin{equation}\label{gh_tower_global}
	\mathcal{M}_\infty \to \dotsb \to \mathcal{M}_1 \to \mathcal{M}_0
\end{equation}
converges globally, meaning that the moduli space $\mathcal{M}_\infty$ of potential $\infty$-stages for all $E$-local spectra is the limit of the moduli stages $\mathcal{M}_n$ of potential $n$-stages. As we saw, this convergence is a consequence of the completeness of the $\infty$-category of synthetic spectra based on $E$; in turn, this follows from an algebraic fact, namely the existence of a horizontal vanishing line in the $E_2$-page of the $E$-based Adams spectral sequence for the Hopkins-Ravenel finite spectrum.

If $E$ is replaced with a general Adams-type homology theory, in general none of the above needs to be true. In fact, completeness fails already for synthetic spectra based on $H$, as a consequence of the existence of an $H$-local but not $H$-nilpotent complete spectrum. We couldn't find an explicit example in the literature, and so we include one here, due to Robert Burklund, see \cref{example:non_vanishing_adams_spectral_sequence}. 

A silver lining on the completness issue is that in solving a given realization problem, one cares less about the Goerss-Hopkins tower in its entirety than about its relativization
\begin{equation}\label{gh_tower_local}
	\mathcal{M}_\infty^M \to \dotsb \to \mathcal{M}_1^M \to \mathcal{M}_0^M = \{M\},
\end{equation}
where $M$ is some fixed $E_{*}E$-comodule and 
\[	\mathcal{M}_n^M = \mathcal{M}_n \times_{\mathcal{M}_0} \{M\}	\]
is the moduli of potential $n$-stages for a realization of $M$. The theorem we will actually prove, which by \cref{cor:if_c_is_complete_then_we_have_an_obstruction_theory_to_constructing_a_periodic_a_module} will have a homological variant of \cref{thm:toda_obstruction_theory} as an immediate consequence, says that the tower \eqref{gh_tower_local} converges whenever $M$ is a comodule that is bounded below.

Again, the argument relies on the existence of vanishing lines on the $E_2$-page of the Adams spectral sequence. In the case of Morava $E$-theory, these vanishing lines were horizontal, as in  \cref{lemma:for_morava_e_theory_any_fp_covered_by_one_with_homology_of_finite_homological_dimension}. In the case of ordinary mod $p$ homology, the vanishing lines are of slope $\frac{1}{q}$ where $q=2p-2$, and are due to Adams. 

Note that we work with the Eilenberg-MacLane spectrum for definiteness, but the same arguments would establish convergence of the Goerss-Hopkins tower whenever the Adams $E_{2}$-term has vanishing lines of positive slope. Most notably, the convergence also holds in the bounded below case for $BP$, see \cref{rem:toda_obstruction_theory_for_bp}. In fact, the argument in the latter case is easier, since the Adams-Novikov $E_{2}$-term has a vanishing line already in the case of the sphere.

If $M$ is a finite $H_{*}H$-comodule, then it is the $\mathbb{F}_{p}$-linear dual of a module over the Steenrod algebra. The latter contains the Bockstein element $\beta \in A^{1}$ with the property that $\beta^{2} = 0$, which subsequently acts on $M$. We will say that a finite $H_{*}H$-comodule $M$ is \emph{freely acted on by the Bockstein} if it free as a module over the exterior algebra $\mathbb{F}_{p}[\beta] / \beta^{2}$.

\begin{thm}[Adams vanishing]
\label{thm:adams_vanishing}
Let $M$ be a finite comodule freely acted on by the Bockstein and let $N$ be bounded from below. Then, $\Ext_{H_{*}H}^{s, t}(M, N)$ has a vanishing line of slope $\frac{1}{q}$.

More precisely, we have $\Ext_{H_{*}H}^{s, t}(M, N) = 0$ for $s > \frac{1}{q}(t-s) + k$, where $k = b(N)-t(M)-3$ with $b(N)$ being the degree of the bottom class in $N$ and $t(M)$ the degree of the top class in $M$. 
\end{thm}

\begin{proof}
Since $\Ext^{s, t}(M, -)$ commutes with filtered colimits, because it can be computed by the cobar complex, without loss of generality we can assume that $N$ is finite, too. Observe that by linear duality we have an isomorphism 

\begin{center}
$\Ext^{s, t}_{H_{*}H}(M, N) \simeq \Ext^{s, t}_{H_{*}H}(\mathbb{F}_{p}, M^{\vee} \otimes N)$,
\end{center}
where $M^{\vee} \otimes N$ is an $H_{*}H$-comodule which is freely acted on by the Bockstein and is concentrated in degrees $d$ satisfying $d \geq b(N) - t(M)$. 

Thus, by setting $O = (M^{\vee} \otimes N)[t(M) - b(N)]$ we see that to prove the needed statement it is enough to know that $\Ext^{s, t}_{H_{*}H}(\mathbb{F}_{p}, O)$ vanishes for $s > \frac{1}{q}(t-s)-3$ if $O$ is a finite comodule freely acted on by the Bockstein and concentrated in non-negative degrees, which is classical, see \cite[3.4.5]{ravenel_complex_cobordism}. 
\end{proof}
As usual when working with synthetic spectra, statements about comodules alone are not enough, as we also need some control over finite spectra and their homology. 

\begin{lemma}
For every finite $p$-complete spectrum $P$, there exists a finite $p$-complete spectrum $Q$ together with an $H_{*}$-surjection $Q \rightarrow P$ such that $H_{*}Q$ is freely acted on by the Bockstein. 
\end{lemma}

\begin{proof}
Take $Q := P \wedge (\Sigma^{-1} S^{0} / p)$, with the map $Q \to P$ induced by the projection $\Sigma^{-1} S^{0} / p \rightarrow S^{0}$.
\end{proof}

\begin{lemma}
\label{lemma:bounded_subspectrum_also_freely_acted_on_by_the_bockstein}
Let $P$ be a finite $p$-complete spectrum such that $H_{*}P$ is freely acted on by the Bockstein and let $k \in \mathbb{Z}$. Then, there exists a finite $p$-complete spectrum $Q$ such that $H_{*}Q$ is freely acted on by the Bockstein and concentrated in degrees $d \leq k+1$, together with a map $Q \rightarrow P$ such that $H_{*}Q \rightarrow H_{*}P$ is an isomorphism in degrees $d \leq k$. 
\end{lemma}

\begin{proof}
To fix notation, let $M \hookrightarrow H_{*}P$ be the smallest $\mathbb{F}_{p}$-linear subspace freely acted on by the Bockstein which contains all elements of degree $d \leq k$. Observe that since $\mathbb{F}_{p}[\beta] / \beta^{2}$ is concentrated in only two degrees, $M$ has no elements in degrees above $k+1$. 

Since $P$ is a finite $p$-complete spectrum, it admits a minimal cell structure where the cells are in one-to-one correspondence with any chosen basis of $H_{*}P$ over $\mathbb{F}_{p}$, see \cite[Theorem 3.3]{baker2004minimal}. We can choose a basis for $M_{k+1}$, extend it to a basis of $H_*P$, and consider the corresponding cell structure on $P$. Then, $M$ can be realized as the homology of a subcomplex of $P$, namely of the $k$-skeleton together with those $(k+1)$-cells which correspond to basis elements of $M$. This subcomplex is the sought after finite spectrum $Q$.  
\end{proof}

The following is the main result of this section.

\begin{thm}
\label{thm:goerss_hopkins_tower_converges_for_bounded_below_hh_comodules}
Let $M$ be a $H_{*}H$-comodule bounded from below. Then, the Goerss-Hopkins tower for realizations of $M$ converges; that is, the diagram
\begin{equation}\label{eq:goerss_hopkins_tower_over_M}
\mathcal{M}_{\infty}^{M} \rightarrow \ldots \rightarrow \mathcal{M}_{1}^{M} \rightarrow \mathcal{M}_{0}^{M},
\end{equation}
where $\mathcal{M}_{n}^{M} := \mathcal{M}_{n} \times _{\mathcal{M}_{0}^{M}} \{ M \}$, is a limit diagram of $\infty$-categories. 
\end{thm}

\begin{proof}
Recall from \cite[5.5.6.23]{lurie_higher_topos_theory} that a Postnikov pretower is a diagram
\begin{center}
$\ldots \rightarrow X_{n} \rightarrow \ldots \rightarrow X_{1} \rightarrow X_{0}$
\end{center}
such that each map $X_m \to X_n$, for $m \ge n$, induces an equivalence $(X_{m})_{\leq n} \simeq X_n$. Then the limit of the tower \eqref{eq:goerss_hopkins_tower_over_M} can be identified with the $\infty$-category of Postnikov pretowers $(X_n)$ in $\synspectra$, such that each $X_n$ is a potential $n$-stage for an $H$-local spectrum realizing $M$.

To prove the statement, it is enough to show that any such pretower can be extended to a Postnikov tower $X \rightarrow \ldots \rightarrow X_{1} \rightarrow X_{0}$, which means that $X \rightarrow X_{m}$ induces an equivalence $X_{\leq m} \simeq X_{m}$ for each $m$. Given such an extension, we necessarily have $X \in \mathcal{M}_{\infty}^{M}$ by \cref{rem:if_ambient_category_complete_infinite_stages_a_limit_of_finite_ones}. If such an $X$ exists, then because $\synspectra$ is separated, $X \simeq \varprojlim X_{n}$. Thus, all we have to verify is that the limit of the pretower has the needed property.

Let us fix a Postnikov pretower $(X_{n})$ as above. We claim that if $P$ is a finite $p$-complete spectrum such that $H_{*}P$ is freely acted on by the Bockstein, then the tower 

\begin{center}
$\ldots \rightarrow [\Sigma^{k} \nu P, X_{1}] \rightarrow [\Sigma^{k} \nu P, X_{0}]$,
\end{center}
of homotopy classes of maps in $\synspectra$ stabilizes; in other words, there exists an integer $N$ such that $[\Sigma^{k} \nu P, X_{n}] \simeq [\Sigma^{k} \nu P, X_{N}]$ for $n \geq N$, which implies that $[\Sigma^{k} \nu P, X] \simeq [\Sigma^{k} \nu P, X_{n}]$ for $X := \varprojlim X_{n}$. Moreover, we claim that the bound $N$ depends only on $k$ and $t(H_{*}P)$, the degree of the top class in the homology of $P$. 

Arguing as in the proof of \cref{thm:synthetic_spectra_based_on_morava_e_theory_complete}, let $X_{[n]}$ denote the fibre of $X_{n} \rightarrow X_{n-1}$, observe that we have an equivalence $X_{[n]} \simeq \Sigma^{n} M[-n]$. Since the latter is canonically a $\monunitt{0}$-module, being a suspension of a discrete synthetic spectrum, we have isomorphisms 

\begin{center}
$\pi_{j} \map(\nu P, X_{[n]}) \simeq \Ext^{n, n-j}(H_{*}P, M)$. 
\end{center}
By Adams vanishing, which we stated as \cref{thm:adams_vanishing}, for any fixed $j$ this vanishes for large values of $n$, with the precise vanishing range only depending on the degree $t(B)$ of the bottom class of $M$, which is fixed, and $t(H_{*}P)$, as promised. Then, the long exact sequence of homotopy implies the needed stabilization. 

Let $H_{\alpha}$ be a filtered diagram of finite $p$-complete spectra such that $\varinjlim H_{\alpha} \simeq H$ in the $H$-local category, this can be obtained by $p$-completing the skeletal filtration of $H$. By \cite[4.17]{pstrkagowski2018synthetic}, any such an filtered diagram yields an explicit formula 

\begin{center}
$(\pi_{k} X)_{l} \simeq \varinjlim [\Sigma^{k} \nu (\Sigma^{l} DH_{\alpha}), X]$,
\end{center}
where by the subscript $l$ we mean the subgroup of degree $l$ elements in the $H_{*}H$-comodule $\pi_{k}X$, and $D$ refers to the $H$-local Spanier-Whitehead dual. Here, as elsewhere in the paper, by $\pi_{*}$ we denote the homotopy groups in the sense of prestable $\infty$-categories, which are referred to in \cite{pstrkagowski2018synthetic} as the $t$-structure homotopy groups.

Filtering $H$ by skeleta as in the proof of \cref{lemma:bounded_subspectrum_also_freely_acted_on_by_the_bockstein}, we may assume that each of $H_{*} H_{\alpha}$ is concentrated in non-negative degrees and freely acted on by the Bockstein. If that is the case, then the top class of $H_{*}(\Sigma^{l} DH_{\alpha})$ is in degree $l$. It then follows from the explicit bound we gave on the stabilization of homotopy classes of maps that once we fix $k$ and $l$, then there exists an $N$ such that 

\begin{center}
$[\Sigma^{k} \nu (\Sigma^{l} DH_{\alpha}), X] \simeq [\Sigma^{k} \nu (\Sigma^{l} DH_{\alpha}), X_{n}]$
\end{center}
for $n \geq N$ and all $\alpha$. Passing to the filtered colimit in $\alpha$, and using the explicit formula above, this yields $(\pi_{k} X)_{l} \simeq (\pi_{k} X_{n})_{l}$ for all large enough $n$. 

Now, let $m \geq 0$, to show that $X_{\leq m} \simeq X_{m}$ we have to check that $\pi_{k} X \simeq \pi_{k} X_{m}$ for $k \leq m$. Since isomorphism of comodules are degreewise, it is enough to verify that $(\pi_{k} X)_{l} \simeq (\pi_{k} X_{m})_{l}$ for all $l$. By what we proved above, one can choose an $n \geq m$ such that $(\pi_{k} X)_{l} \rightarrow (\pi_{k} X_{n})_{l}$ is an isomorphism. Since the same is true for $(\pi_{k} X_{n})_{l} \rightarrow (\pi_{k} X_{m})_{l}$, the composite must hold for the composite, which is exactly what we wanted to show. This ends the argument. 
\end{proof}

\begin{cor}
\label{cor:homological_toda_obstruction_theory}
Let $H = H \mathbb{F}_{p}$ be the Eilenberg-MacLane spectrum and let $M$ be a $H_{*}H$-comodule which is bounded below. Then, there exists a sequence of inductively defined obstructions

\begin{center}
$\theta_{n} \in \Ext_{\Comod_{H_{*}H}}^{n+2, n}(M, M)$, where $n \geq 1$
\end{center}
which vanish if and only if there exists a spectrum $X$ such that $H_{*}X \simeq M$ as comodules.
\end{cor}

\begin{proof}
Taking into account \cref{thm:goerss_hopkins_tower_converges_for_bounded_below_hh_comodules}, this follows immediately from \cref{cor:if_c_is_complete_then_we_have_an_obstruction_theory_to_constructing_a_periodic_a_module}. 
\end{proof}

\begin{rem}
\label{rem:toda_obstruction_theory_for_bp}
The analogue of \cref{thm:goerss_hopkins_tower_converges_for_bounded_below_hh_comodules} also holds for the Brown-Peterson spectrum; that is, the Goerss-Hopkins tower converges for any bounded below $BP_{*}BP$-comodule $M$. It follows that the analogue of \cref{cor:homological_toda_obstruction_theory} is true as well; that is, there are obstructions in $\Ext^{n+2, n}_{BP_{*}BP}(M, M)$ to realizing $M$ as a $BP$-homology of a spectrum.

To see this, notice that the key argument in the proof of \cref{thm:goerss_hopkins_tower_converges_for_bounded_below_hh_comodules} is that the Eilenberg-MacLane spectrum can be written as a filtered colimit $H \simeq H_{\alpha}$ of finite spectra with the property that for any fixed $j$ 

\begin{center}
$\Ext^{n, n-j}_{H_{*}H}(H_{*}(DH_{\alpha}), M)$
\end{center}
vanishes for large $n$. Moreover, the bound on vanishing depends only on the bounded below comodule $M$, in fact on the degree of its bottom class, but not on $\alpha$. 

In the case of $BP$, one can use the skeletal filtration, since the skeleta of $BP$ are all finite projective with no cells in negative degrees. To prove the needed vanishing one argues as in the proof of \cref{thm:adams_vanishing}, reducing to the case of showing that $\Ext^{s, t}_{BP_{*}BP}(BP_{*}, O)$ has a vanishing line of fixed positive slope for any finitely generated $BP_{*}BP$-comodule $O$, with the intercept of that line depending only on the degree of the bottom class of $O$. 

It is classical that such a line of slope $1$ exists when $O = BP_{*}$, see \cite{ravenel_complex_cobordism}[5.1.23], and the general case follows by induction on the Landweber filtration of $O$ \cite{hovey2003comodules}. 
\end{rem}

\begin{example}[Burklund]
\label{example:non_vanishing_adams_spectral_sequence}
The following is an example of an $H$-local, but not $H$-nilpotent complete spectrum $X$ due to Robert Burklund. It follows that $\nu X$ is a hypercomplete synthetic spectrum which is not a limit of its Postnikov tower, see \cite[5.6, 5.4]{pstrkagowski2018synthetic}, \cite[A.11]{burklund2019boundaries}. This shows that the $\infty$-category of connective, hypercomplete synthetic spectra based on $H$ is not Postnikov complete. 

Before we move on, let us recall some terminology. Following Bousfield, if $X$ is a spectrum, we call the limit of the Adams resolution

\begin{center}
$H \otimes X \rightrightarrows H \otimes H \otimes X \triplerightarrow \ldots$
\end{center}
the \emph{$H$-nilpotent completion} of $X$ and denote it by $(X)_{H}^{\wedge}$. We say $X$ is \emph{$H$-nilpotent complete} if $X \rightarrow (X)_{H}^{\wedge}$ is an equivalence, that is, if $X$ is a limit of its Adams tower. Any $H$-nilpotent complete spectrum is $H$-local, and it is a classical result of Bousfield that the converse is true for bounded below spectra. 

Since $H$-localization doesn't change $H$-homology, to produce a counterexample it is enough to find a spectrum $X$ such that the natural map $X \rightarrow (X)_{H}^{\wedge}$ is not an $H$-equivalence. If this is the case, then the Bousfield localization $X_{H}$ is the sought after counterexample.

Let $k(n)$ denote the connective Morava $K$-theory with $k(n)_{*} \simeq \mathbb{F}_{p}[v_{n}]$, where $|v_{n}| = 2p^{n}-2$. It is classical that each of these is an $H$-nilpotent complete spectrum with collapsing Adams spectral sequence, and $v_{n}^{k}$ in Adams filtration $k$. We define the counterexample as
\[ X := \bigoplus_{i>0} \Sigma^{ - |v_{n_i}| i } k(n_i), \]
where $\{n_i\}$ is any sequence of natural numbers increasing rapidly enough so that
\[ |v_{n_i}| > |v_{n_{i-1}}|  (i-1) \]
for all $i > 0$. The reader can easily verify that this condition guarantees that
\begin{enumerate}
\item $\pi_0X$ is the direct sum $\bigoplus  _{i > 0} \mathbb{F}_p\{v_{n_i}^i\}$ and 
\item $\pi_{k} X$ is a finite-dimensional $\mathbb{F}_{p}$-vector space for $k \neq 0$.
\end{enumerate}

Since a sum of $H$-Adams resolutions is an $H$-Adams resolution, the nilpotent completion $(X)_{H}^{\wedge}$ is a limit of the direct sum of the $H$-Adams resolutions of the $\Sigma^{-|v_{n_{i}}|i} k(n_{i})$. It follows that the induced spectral sequence of signature
\[
\bigoplus_{i > 0} \textnormal{Ext}^{s, t}_{H_{*}H}(H_{*}, H_{*}(\Sigma^{-|v_{n_{i}}|i} k(n_{i}))) \Rightarrow \pi_{t-s} (X)_{H}^{\wedge}
\]
is conditionaly convergent in the sense of Boardman. This spectral sequence collapses, and our conditions on $n_{i}$ guarantee that $E_{2}^{s, t}$ is a finite abelian group in any bidegree $(s, t)$, yielding strong convergence as all of the relevant $\varprojlim^{1}$-terms vanish. 

It follows that $\pi_{k} X \rightarrow \pi_{k} (X)_{H}^{\wedge}$ is an isomorphism for $k \neq 0$ and the map $\pi_{0} X \rightarrow \pi_{0} (X)_{H}^{\wedge}$ can be identified with completion with respect to the Adams filtration. Thus
\[
\pi_{0} (X)_{H}^{\wedge} \simeq \varprojlim \pi_{0} X / F^{k} \pi_{0} X \simeq \varprojlim \bigoplus _{k > i > 0} \mathbb{F}_{p} \{ v_{n_{i}}^{i} \} \simeq \prod_{i > 0} \mathbb{F}_{p} \{ v_{n_{i}}^{i} \}.
\]
It follows that the cofibre of $X \rightarrow (X)_{H}^{\wedge}$ is an Eilenberg-MacLane spectrum associated to the quotient  $\mathbb{F}_{p}$-vector space $\prod \mathbb{F}_{p} \{ v_{n_{i}}^{i} \} / \bigoplus \mathbb{F}_{p} \{ v_{n_{i}}^{i} \}$. Thus, the cofibre has non-vanishing $H$-homology, so that $X \rightarrow (X)_{H}^{\wedge}$ is not an $H$-localization, as claimed. 
\end{example}

\bibliographystyle{amsalpha}
\bibliography{abstract_gh_theory_bibliography}

\end{document}